\documentclass[11pt, a4paper]{article}
\usepackage[utf8]{inputenc} 
\usepackage[a4paper, margin=2.5cm]{geometry}
\usepackage{parskip}
\usepackage{comment}
\usepackage{enumitem}
\usepackage{multirow}
\usepackage{cite}
\usepackage{mathtools} 
\usepackage{amsthm}
\usepackage{amssymb}
\usepackage{bbm}
\usepackage{mathrsfs}
\usepackage{accents}
\usepackage{stackengine}
\usepackage{scalerel}

\mathtoolsset{showonlyrefs}

\usepackage{xcolor}
\usepackage[export]{adjustbox}
\usepackage{tikz,pgfplots}
\usetikzlibrary{
  arrows,
  arrows.meta,
  bending,
  calc,
  decorations.markings,
  decorations.text,
  intersections,
  patterns,
  spy
}
\pgfplotsset{compat=1.18}

\usepackage{algorithm}
\usepackage{algorithmic}

\definecolor{hanblue}{rgb}{0.27, 0.42, 0.81}
\definecolor{mordantred19}{rgb}{0.68, 0.05, 0.0}
\definecolor{myred}{rgb}{0.8, 0.0, 0.0}
\definecolor{mygreen}{rgb}{0.0, 0.5, 0.0}

\usepackage[
  colorlinks,
  citecolor=hanblue,
  linkcolor=mordantred19
]{hyperref}

\DeclareMathAlphabet{\mathpgoth}{OT1}{pgoth}{m}{n}

\DeclareFontFamily{U}{mathx}{}
\DeclareFontShape{U}{mathx}{m}{n}{<-> mathx10}{}
\DeclareSymbolFont{mathx}{U}{mathx}{m}{n}
\DeclareMathAccent{\widehat}{0}{mathx}{"70}
\DeclareMathAccent{\widecheck}{0}{mathx}{"71}

\tikzset{
  mid arrow/.style={
    postaction={
      decorate,
      decoration={
        markings,
        mark=at position 0.5 with {\arrow{Stealth[length=4pt,width=5pt]}}
      }
    }
  }
}

\DeclareMathOperator*{\argmin}{arg\,min}

\DeclareMathOperator{\supp}{supp}
\DeclareMathOperator{\loc}{loc}

\DeclareMathOperator{\Id}{Id}

\newcommand{\mres}{\mathbin{\vrule height 1.4ex depth 0pt width 0.13ex\vrule height 0.13ex depth 0pt width 1.0ex}}

\newcommand{\calM}{\mathcal{M}}
\newcommand{\calI}{\mathcal{I}}
\newcommand{\calF}{\mathcal{F}}

\newcommand{\calP}{\mathcal{P}}
\newcommand{\calB}{\mathcal{B}}
\newcommand{\calR}{\mathcal{R}}
\newcommand{\calD}{\mathcal{D}}

\newcommand{\calU}{\mathcal{U}}
\newcommand{\calL}{\mathcal{L}}
\newcommand{\calQ}{\mathcal{Q}}

\newcommand{\bS}{\mathbb{S}}
\newcommand{\bC}{\mathbb{C}}

\newcommand{\one}{\mathbbm{1}}
\renewcommand{\i}{\mathrm{i}}

\newcommand{\R}{\mathbb{R}}
\newcommand{\Z}{\mathbb{Z}}

\newcommand{\N}{\mathbb{N}}

\newcommand{\M}{\mathcal{M}}

\newcommand{\weakstar}{\stackrel{*}{\rightharpoonup}}
\newcommand{\weak}{\rightharpoonup}

\newcommand{\dd}{\,\mathrm{d}}

\newcommand{\mathdiam}[1]{\accentset{\diamond}{#1}}

\renewcommand{\phi}{\varphi}

\NewDocumentCommand{\narrowcheck}{m}{%
  \tikz[baseline=(X.base)]{
    \node[inner sep=0.5pt] (X) {$#1$};
    \draw[line width=0.1ex]
      ([yshift=0.4ex,xshift=-0.15em]X.north) --
      ([yshift=0.1 ex,xshift=0.00em]X.north) --
      ([yshift=0.4ex,xshift=0.15em]X.north);
  }
}

\newcommand\restr[2]{{
  \left.\kern-\nulldelimiterspace
  #1
  \vphantom{\big|}
  \right|_{#2}
}}

\theoremstyle{plain}
\newtheorem{theorem}{Theorem}[section]

\newtheorem{lemma}[theorem]{Lemma}
\newtheorem{proposition}[theorem]{Proposition}
\newtheorem{corollary}[theorem]{Corollary}

\theoremstyle{definition}
\newtheorem{definition}[theorem]{Definition}

\theoremstyle{remark}

\newtheorem{example}[theorem]{Example}

\title{
Towards sparse optimization over convex loops: Equivalence of Square Root Velocity distance and Wasserstein-Fisher-Rao}
\author{Giacomo Cristinelli$^\ast$, Jos\'e A. Iglesias\thanks{Department of Applied Mathematics, University of Twente, 7500AE Enschede, The Netherlands \newline (\texttt{g.cristinelli@utwente.nl, jose.iglesias{@}utwente.nl})}}
\date{}

\begin{document}

\maketitle
\begin{abstract}
    The Wasserstein--Fisher--Rao (WFR) distance on $S^{2}$ has recently been shown to coincide with a classical elastic distance between $S^{2}$-immersions in the theory of Riemannian shape analysis.
    While this correspondence holds in dimension $2$, the analogous statement fails in general on $S^{1}$ and, in the case of convex curves, it cannot be derived from existing two-dimensional arguments.
    In this paper, we establish that for convex absolutely continuous immersions of $S^{1}$ in the plane, the shape distance induced by the square root velocity transformation (SRVT) is indeed equivalent to the WFR distance acting on their associated length measures. The proof exploits a monotonicity principle for optimal transport on the universal cover of the circle, which in turn guarantees the existence of an optimal reparametrization achieving the SRVT infimum and enables a one-dimensional unbalanced optimal transport reformulation.
    Motivated by this equivalence, we further investigate the role of sparsity in shape optimization problems formulated in terms of length measures and regularized by the WFR distance. We study linear optimization over the corresponding balls, for which we prove a finiteness result when the reference measure is discrete, and propose a convex, positively one-homogeneous regularizer suitable for conditional gradient algorithms.
\end{abstract}
 
\normalsize
\section{Introduction}
Elastic shape analysis provides a differential–geometric framework for comparing curves and surfaces modulo reparametrization, with applications ranging from computer vision and medical imaging to statistics and inverse problems. In this setting, shapes are typically modeled as immersions equipped with reparametrization-invariant Riemannian metrics, and shape distances are defined as geodesic distances in infinite-dimensional manifolds. Specifically, these are formulated by identifying the corresponding tangent spaces with deformation fields and using a first order Sobolev metric on these, possibly weighting differently the contributions on directions tangential and orthogonal to the base object. While this approach is conceptually natural, it often leads to highly nonconvex optimization problems that are analytically delicate and computationally demanding.

A major breakthrough in the field of one-dimensional elastic shape analysis was the introduction of the square root velocity transformation (SRVT), which maps an immersed curve to a representation in a flat Hilbert space and was shown in \cite{SriKlaJosJer11} to turn Sobolev-type elastic metrics with specific weights into the standard $L^2$ metric. This transformation allows distances to be computed more efficiently, by optimizing over a class of reparametrizations. Over the past decade, the SRVT has become a central tool in elastic shape analysis, both theoretically and numerically. 
An extension to surfaces was also proposed, called square root normal field (SRNF) distance \cite{SriSRNF}.

More recently, a connection has emerged between elastic shape analysis and unbalanced optimal transport, and, in particular, the Wasserstein–Fisher–Rao (WFR) distance, also known as the Hellinger–Kantorovich distance, which extends classical optimal transport by allowing for mass creation and destruction. In the case of $S^2$-immersions, it has been shown that the elastic distance induced by the square root normal field (SRNF) distance coincides exactly with the WFR distance between suitably defined surface area measures \cite{BauHarKlas22, HaBaKla24}. This result provides a striking reinterpretation of elastic shape matching as an optimal transport problem and offers new analytical and algorithmic perspectives. However, in this two-dimensional framework, it remains computationally challenging to reconstruct the unique convex surface (up to translation) associated with such a measure \cite{Schneider1993ConvexSurfaces, sellaroli2017algorithmreconstructconvexpolyhedra}.

This difficulty does not arise in dimension one, where the measure-to-(convex-)curve correspondence admits an explicit analytical description. However, the geometry of $S^1$ immersions \cite{michor2006riemannian, bauer2024elastic} differs fundamentally from that of surfaces \cite{MichorSRNF}. In fact, the equivalence between SRVT-based elastic distances and WFR distances fails in general for closed curves. 
Even when restricting attention to convex planar curves and their length measures \cite{charon2020lengthmeasuresplanarclosed}, proving the equivalence is not trivial, and the existing two-dimensional arguments break down. Indeed, in the one-dimensional setting, the reparametrization group, naturally sought inside $W^{1,1}(S^1,S^1)$, is not dense in the space of piecewise continuous bijections with respect to any topology in which the SRVT functional is continuous: removing finitely many points disconnects $S^1$ but not $S^2$, so a piecewise constant reparametrization of the circle has no absolutely continuous approximation with controlled derivative.
Nevertheless, drawing on the extensive literature on optimal transport on $S^1$ \cite{Delon_2010, bonet2023sphericalslicedwasserstein, hundrieser2021statisticscircularoptimaltransport}, it is possible to construct optimal reparametrizations from a transport-map perspective by formulating a suitable Monge problem.

Beyond this theoretical correspondence, the SRVT–WFR equivalence has important implications for optimization and inverse problems involving shapes. In recent work on inverse obstacle scattering \cite{Eckhardt_2019}, elastic energies have been employed as Tikhonov regularizers on shape manifolds, yielding parametrization-invariant formulations with rigorous regularization properties and convergence guarantees. In that setting, the bending energy acts as a geometric stabilizer that penalizes oscillatory or highly curved shapes. 

In light of such applications, we explore some interactions between the WFR distance and techniques of optimization in spaces of measures. One important such direction is distributionally robust optimization, in which one considers maximization problems reflecting the worst case scenario among distributions {close} to an empirical one, most often measured with the Wasserstein distance \cite{MohKuh18,GaoKle23}. 

Another important line of research in optimization with measures towards implementable algorithms is the applicability of sparse optimization methods. A basic intuition and motivation behind these is that they construct the solution progressively by iteratively finding elements of some dictionary. This is the case in fully-corrective generalized conditional gradient methods \cite{BreCarFanWal24} for composite optimization, in which the coefficients of the representation of the iterates are also optimized over, requiring the nonsmooth term of the cost to be one-homogeneous (often a norm). In that case, the insertion step is again a linear optimization problem over the corresponding unit ball. These optimization problems can then be restricted to be over the set of extreme points of said ball, provided that one has a characterization of them. This program has led to efficient optimization methods based on the usual norm for $\M(\Omega)$ and extremals supported on single points both with discretization on- or off-the-grid \cite{DenDuvPeySou20}, but also to other more involved regularizers where extremals are supported on open curves \cite{BreCarFanRom23,LavBlaAub24} and the total (gradient) variation where they arise from the boundaries of simple sets, both with off-the-grid polygonal parametrizations \cite{CasDuvPet22} and on simplicial meshes for PDE-constrained optimization \cite{Cristinelli_2025, cristinelli2025linearconvergenceonecutconditional}. The key first result for applying such generalized conditional gradient methods (but also other results such as exact recovery in inverse problems \cite{CarDel23}) would be a characterization of the extreme points of the unit ball. Cases particularly related to our context are those in which such a characterization is achieved for regularizers arising from optimal transport, such as dynamic formulations of both Wasserstein \cite{BreCarFanRom21} and Wasserstein-Fisher-Rao \cite{BreCarFan22} distances resulting {in} measures supported on paths, and unbalanced Kantorovich-Rubinstein norms \cite{CarIglWal25} which take into account differences of mass by an infimal convolution construction and lead to ``dipoles'', signed measures supported on two points with equal size and opposite sign. 

\subsection*{Contribution and outlook}
The first main contribution of this paper is to show that a precise equivalence between the SRVT distance and Wasserstein-Fisher-Rao can be recovered in a natural and geometrically meaningful setting. 
We prove that, for convex absolutely continuous immersions of $S^1$ in the plane, the shape distance induced by the SRVT coincides with the Wasserstein--Fisher--Rao distance acting on the associated length measures on the circle.
This result provides a one-dimensional analogue of the previously known two-dimensional correspondence. In addition, in contrast to the two-dimensional setting, our analysis guarantees the existence of optimal reparametrizations for convex curves within the SRVT framework.

In the second part of the paper, we study sparsity phenomena in optimization problems involving the Wasserstein--Fisher--Rao distance, with a view to shape optimization problems in terms of length measures. Specifically, we analyze linear optimization over WFR balls and prove a finiteness result for extremal solutions when the reference measure is also finitely supported.
We propose in Definition \ref{def: homogenized} a ``homogenized'' version $\bS_\nu(\cdot)$ of the WFR energy, again with respect to a fixed reference measure $\nu$. It compares the measures rescaled by their total mass, which also has the geometric interpretation of adding scale invariance, which could in itself be desirable in applications where the fixed reference measure reflects a prior on desired shapes, but not sizes, of the results.

 To conclude, we {provide a preliminary result on the structure of }the extreme points arising from $\bS_\nu$ with $\nu$ finitely supported, while leaving the full characterization for future work. {Such a characterization could facilitate the development of} algorithms in which the iterates are themselves convex loops, and the reconstruction is performed by progressively adding vertices to them to minimize problems of the type
\begin{equation}
    \min_{\mu\in \calM^0(S^1)} \calF (K\mu) + \bS_\nu(\mu),
\end{equation}
where $\calM^0(S^1)$ is the space of Borel measures on $S^1$ with vanishing first moment, $Y$ is a Hilbert space, $K:\calM^0(S^1)\to Y$ is an operator reflecting the effect of the geometry on measurements, and $\calF: Y\to [0,+\infty]$ is a fidelity term. We see all of the results presented in this paper as a first step towards such effective computational methods.

More generally, the results of the present paper suggest a complementary and measure-theoretic viewpoint on elastic regularization strategies. By identifying the SRVT elastic distance with the Wasserstein--Fisher--Rao distance on length measures, we provide an alternative formulation of elastic Tikhonov regularization directly at the level of measures. From this perspective, regularization by elastic energies can be interpreted as penalizing local stretching or compression. This interpretation is particularly appealing for inverse problems, as it opens the door to sparsity-promoting methods and convex optimization techniques that are difficult to access in the classical manifold-based framework.

\subsection*{Organization of the paper}

In Section \ref{sec: pre} we introduce the square root velocity transformation framework and recall the logarithmic entropy transport formulation of the WFR unbalanced optimal transport distance used throughout, together with its main structural properties and compactness results. 

Section \ref{sec: equi} establishes the core theoretical result: the identification of the SRVT distance on convex planar loops with the corresponding unbalanced transport metric between their length measures. In Section \ref{sec: adms} we consider the admissibility of length measures to be able to apply existing results on the logarithmic entropy transport point of view of WFR, to then discuss their periodic lifting in Section \ref{sec: lift} and complete the equivalence proof in Section \ref{sec: equiproof}.

In Section \ref{sec: opti} we turn to the optimization perspective. Section \ref{sec: linearopt} reformulates a linear maximization over WFR balls around a discrete measure as a finite-dimensional problem on couplings and derives sparsity results for extremal solutions, including the case with additional moment constraints. Finally, in Section \ref{sec: homogenized} we introduce a one-homogeneous version of the WFR distance with a view to sparse optimization methods and {demonstrate the complexity of the extreme-point structure of its unit ball}.

{In Section \ref{sec: numer}, we provide numerical experiments that confirm Theorem \ref{thm: elastic-WFR} on three examples and show that computing WFR on length measures is generally more efficient than optimizing the reparametrization.}

Technical arguments and auxiliary results which closely follow existing literature, in particular concerning the structure of optimal couplings and extended-valued cost functions, are collected in Appendix \ref{sec: app1} {and \ref{sec: app2}}.

\subsection*{Notation and standing assumptions}
Throughout the paper we work on a Polish metric space $(X,d)$ (typically $X=S^1$ or $X=\R,\bC$). 
We denote by $\mathcal{M}(X)$ the space of finite signed Radon measures on $X$ and by 
$\mathcal{M}_+(X)$ the cone of nonnegative measures.  
For $\mu\in\mathcal{M}(X)$, we write $|\mu|$ for its total variation measure and 
$\operatorname{supp}\mu$ for its support, defined as the smallest closed set $F\subset X$ such that 
$|\mu|(X\setminus F)=0$. We denote the one-dimensional Lebesgue measure on $\R$ by $\dd t$ and the Hausdorff measure on $S^1$ by $\dd s$, also outside integrals.
Dirac masses are written $\delta_x$ for $x\in X$.  
The set of nonnegative point masses is
\begin{equation}
    \Delta_+(X):=\{ r\delta_x : r\ge 0,\ x\in X \},
\end{equation}
and the cone of nonnegative finitely supported measures is
\begin{equation}
    \mathcal{D}_{+}(X)
    :=\bigcup_{n=1}^\infty \mathcal{D}_+^n,
    \qquad 
    \mathcal{D}_+^n
    :=\left\{
        \sum_{i=1}^n r_i \delta_{x_i} :
        r_i\ge 0,\ x_i\in X
    \right\}.
\end{equation}
We denote by $\mathcal{P}(X)\subset\mathcal{M}_+(X)$ the set of probability measures and by
\begin{equation}
    \mathcal{P}_{+,f}^n(X)
    := \mathcal{P}(X)\cap \mathcal{D}_+^n(X)
\end{equation}
the set of finitely supported probability measures with exactly $n$ atoms.
Weak-* convergence arising from the Riesz-Markov characterization of $\mathcal{M}(X)$ as a dual space is denoted by $\mu_n \weakstar \mu$ and means
\begin{equation}
    \lim_{n\to\infty}\int_X \varphi\, \dd\mu_n = \int_X \varphi\, \dd\mu
    \qquad\text{for all } \varphi\in C_0(X).
\end{equation}
Narrow convergence of measures is denoted by $\mu_n\weak\mu$ and means
\begin{equation}
    \lim_{n\to\infty}\int_X \varphi\, \dd\mu_n = \int_X \varphi\, \dd\mu
    \qquad\text{for all } \varphi\in C_b(X).
\end{equation}

If $T:X\to Y$ is a measurable map and $\mu\in\mathcal{M}(X)$, the push-forward measure 
$T_\#\mu\in\mathcal{M}(Y)$ is defined by
\begin{equation}
    T_\#\mu(B) := \mu(T^{-1}(B)) 
    \qquad\text{for all Borel sets } B\subset Y.
\end{equation}
When $X\subset\mathbb{R}^n$ we also use the affine subspace of measures with vanishing first moment,
\begin{equation}
    \mathcal{M}^0(X)
    := \left\{\mu\in\mathcal{M}_+(X) :
        \int_{X} x\,\dd\mu(x) = 0\in\mathbb{R}^n
    \right\}.
\end{equation}
For $\mu\in\mathcal{M}_+(X)$ and $1\le p\le\infty$, we denote by $L^p(X,\mu)$ the usual Lebesgue space of $\mu$-measurable functions modulo $\mu$-a.e.\ equality, with norm
\begin{equation}
    \|f\|_{L^p(X,\mu)}
    :=
    \left(
        \int_X |f|^p\,\dd\mu
    \right)^{1/p}
    \quad (1\le p<\infty),
\end{equation}
and the essential supremum norm for $p=\infty$.  
We write
\begin{equation}
    L^p_+(X,\mu)
    :=
    \{ f\in L^p(X,\mu) : f\ge 0 \ \mu\text{-a.e.} \}
\end{equation}
for the positive cone in $L^p(X,\mu)$.

We denote by
\begin{equation}
    \mathcal M_p(X)
    :=
    \left\{
        \mu\in\mathcal M_+(X) :
        \int_X |x|^p\,\dd\mu(x) < \infty
    \right\}
\end{equation}
the space of nonnegative measures with finite $p$-th moment, and
\begin{equation}
    \mathcal M_{p}^0(X)
    :=
    \mathcal M_p(X)\cap \mathcal M^0(X)
\end{equation}
those with finite $p$-th moment and vanishing first moment. 
We say that a family $(\mu_k)\subset\mathcal M_p(X)$ has
uniformly bounded $p$-th moments, for some $p\ge 1$, if
\begin{equation}
    \sup_k \int_X |x|^p\,\dd\mu_k(x) < \infty.
\end{equation}
For measures on $X\times X$, we say that $\gamma\in \mathcal M(X\times X)$ has finite $p$-moment if
\begin{equation}
    \int_{X\times X} \,\dd(x,y)^p\, d\gamma(x,y) < \infty,
\end{equation}
and we write $\mathcal M_p(X\times X)$ accordingly.
Given $\mu_0,\mu_1\in\mathcal M_+(X)$, the set of balanced couplings is
\begin{equation}
    \Gamma_b(\mu_0,\mu_1)
    :=
    \left\{
        \gamma\in\mathcal M_+(X\times X) :
        (\operatorname{proj}_0)_\#\gamma=\mu_0,\;
        (\operatorname{proj}_1)_\#\gamma=\mu_1
    \right\},
\end{equation}
where $\operatorname{proj}_i$ are the coordinate projections.  

We denote by $S_n:=\{\zeta\in \R_+^n: \sum_i^n\zeta_i=1\}$ the standard $n$-simplex, often used for coefficients of convex combinations. Observe that we denote the nonnegative real numbers as $\R_+ = [0,+\infty)$.

These conventions will be used throughout the paper.  
Whenever additional regularity or structural assumptions are required, they are stated explicitly at the corresponding point.

\section{Preliminaries}\label{sec: pre}
\subsection{Square root velocity transformation}

Following the use of the complex square root mapping in \cite{Younes_2008}, the square root velocity transformation was first introduced in \cite{SriKlaJosJer11} in the setting of plane curves as an efficient method for locally computing the elastic distances between smooth curves by reformulating the problem as a flat $L^2$ minimization. More precisely, it defines a local isometry from the space of smooth immersed planar curves, equipped with a specific $1$-Sobolev elastic Riemannian metric, to a flat $L^2$ Hilbert manifold. Generalizations to elastic metrics with general coefficients have been proposed and studied in \cite{Needham_2020} and \cite{bauer2024elastic}.

Let $\calI$ be the space of closed, oriented planar absolutely continuous curves up to translation, i.e.
\begin{equation}
    \calI:=\{c\in W^{1,1}(S^1, \bC): c'(s)\neq 0\quad  a.e.  \quad s\in S^1\}/\sim.
\end{equation}
where $c_1\sim c_2$ in $\calI$ if there exists $z_0\in \bC$ such that the images satisfy $c_1(s)=c_2(s)+z_0$ for a.e. $s\in S^1$ and $c_1$ and $c_2$ have the same orientation.
Absolute continuity as regularity for the preshape space is a natural choice \cite{Bru16, bauer2024elastic}, because, in the setting of open curves, the space of AC immersions represents the metric completion of the geodesic distance with respect to any $1$-Sobolev elastic Riemannian metric on the space of smooth immersions (see \cite[Corollary 2.1]{bauer2024elastic}). 

\begin{definition}
    The square root velocity transformation is the function $\Phi:\calI\to L^2(S^1,\bC\setminus\{0\})$ given by 
\begin{equation}
    \Phi(c):=\frac{c'}{\sqrt{|c'|}}.
\end{equation}
\end{definition}
Note that this is well defined since {it is} invariant with respect to translations of $c$ and
\begin{equation}
    \|\Phi(c)\|^2_{L^2}=\int_{S^1}\frac{1}{|c'|}|c'|^2\dd s=\int_{S^1}|c'|\dd s  <\infty.
\end{equation}

We now define a parametrization-invariant distance on $\calI$ through the action by precomposition of absolutely continuous reparametrizations: 
\begin{equation}
    \Pi(S^1):=\{\phi\in W^{1,1}(S^1,S^1): \phi^{-1}\in W^{1,1}(S^1),\, \phi'(s)>0 \; \text{for a.e.}\; s\in S^1\}.
\end{equation}
This action defines the quotient space $\calQ:=\calI/\Pi(S^1)$. 
\begin{definition}
    The square root velocity transformation distance $d_{\operatorname{SRVT}}$ between $c_0,c_1\in \calI$ is
    \begin{equation}
        d^2_{\operatorname{SRVT}}(c_0,c_1):=\inf_{\phi\in \Pi(S^1)}\int_{S^1}|\Phi(c_0)- \Phi(c_1\circ \phi)|^2\dd s.
    \end{equation}
\end{definition}
The existence of an optimal reparametrization is not guaranteed in general and, to the best of our knowledge, remains an open problem for closed curves. However, for open curves, existence results have been established under $C^1$ regularity assumptions, while non-existence has been shown for certain Lipschitz curves; see \cite{Bru16} for details.

\subsection{Length measure for convex planar curves}
\begin{definition}\label{def: length}
    Given a loop $c\in \calI$, we define its length measure $\mu_c\in \M_+(S^1)$ as the pushforward of the arc-length measure $|c'(s)|\dd s$ via the (tangent) Gauss map $T_c:S^1\to S^1$ which is defined by
    \begin{equation}
        T_c(s):=\frac{c'(s)}{|c'(s)|}.
    \end{equation}
    Explicitly,
\begin{equation}\label{eq: l2 to meas}
    \mu_c:= (T_c)_\#(|c'(s)|\dd s )=\int_{T_c^{-1}(\cdot)} |c'(s)| \;\dd s.
\end{equation}
\end{definition}
The measure $\mu_c$ is called length measure, as it returns the length of the subset of points in $\operatorname{Im}(c)$ whose tangent direction lies inside any Borel set $U\subset S^1$. 
\begin{figure}[!ht]
    \centering
    \begin{tikzpicture}[scale=0.8]
    \fill[fill=black!10] (-3,0)--(-3,1)-- (-4.5,2.5)-- (-6,1) to[out=270, in=240]  (-3,0);
        \begin{scope}[every node/.style={sloped,allow upside down}]
        \draw[mid arrow] (-3,0)-- (-3,1);
        \draw[mid arrow] (-3,1)-- (-4.5,2.5);
        \draw[mid arrow] (-4.5,2.5)--(-6,1);
        \draw[mid arrow] (-6,1) to[out=270, in=240] (-3,0);
        \end{scope}
\end{tikzpicture}
\hspace{3cm}
\begin{tikzpicture}
    \filldraw[blue] (3,3.7) circle(1pt)
    (1.5,3.4) circle (1pt)
    (1.5,0.6) circle (1pt);
    \draw[blue] (3, 3)--(3,3.7)
    (2.25, 2.65)--(1.5,3.4)
    (2.25, 1.35)--(1.5,0.6);
    \filldraw[fill=blue!20, draw=blue] (3, 1) to [out=0, in=170] (4,1.1) to [out=-10, in=240] (4.2,1.15) to[out=60, in=280] (4.1, 1.8) to[out=100, in=305] (3.8, 2.6) to (3.71,2.71) arc (405:270:1 and 1);
    \draw (3,2) circle (1cm);
\end{tikzpicture}
    \caption{An example of a loop $c\in\calI$ on the left, and its length measure $\mu_c$ on the right.}
\end{figure}
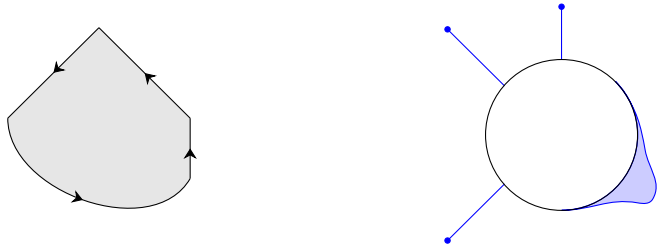
The map $c \mapsto \mu_c$ from the shape space $\mathcal{Q}$ to the space of nonnegative measures $\mathcal{M}^0(S^1)$ is well-defined and surjective. However, it is not injective: each measure $\mu$ admits infinitely many preimages under this map, see Figure \ref{fig:folding} for an example. 
\begin{figure}[b]
    \centering
    \begin{tikzpicture}[scale=0.5]
        \fill[fill=black!10] (0,0)--(5,0)--(5,5)--(0,5)--cycle;
        \begin{scope}[every node/.style={sloped,allow upside down}]
        \draw[mid arrow] (0,0)-- (0,5);
        \draw[mid arrow] (0,5)-- (5,5);
        \draw[mid arrow] (5,5)-- (5,0);
        \draw[mid arrow] (5,0)--  (0,0);
        \end{scope}
    \end{tikzpicture}
    \hspace{3cm}
    \begin{tikzpicture}[scale=0.5]
        \fill[fill=black!10] (0,0)--(5,0)--(5,3)--(3,3)--(3,5)--(0,5)--cycle;
        \begin{scope}[every node/.style={sloped,allow upside down}]
        \draw[mid arrow] (0,0)--  (0,5);
        \draw[mid arrow] (0,5)--  (3,5);
        \draw[mid arrow] (3,5)--  (3,3);
        \draw[mid arrow] (3,3)--  (5,3);
        \draw[mid arrow] (5,3)--  (5,0);
        \draw[mid arrow] (5,0)--  (0,0);
        \end{scope}
    \end{tikzpicture}
    \caption{Two loops with the same corresponding length measure.} 
    \label{fig:folding}
\end{figure}
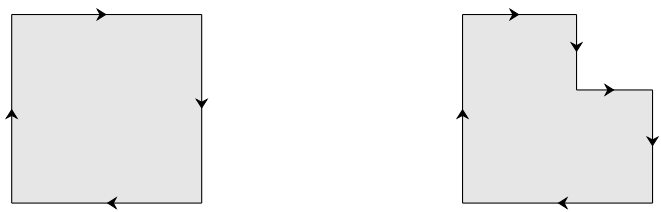
As a result, the measure-theoretic formulation alone is insufficient to capture the full geometric structure of the loops in $\mathcal{I}$.

\begin{definition}
    We say that $c\in \calI$ is a convex loop if its image is the boundary of some convex set in $\bC$. We denote by $\mathdiam{\calI}$ the space of such curves, and we assume that they are all positively oriented, meaning that the interior of the convex set lies on the left-hand side of $c'$ almost everywhere. We denote the quotient space by $\mathdiam{\calQ}=\mathdiam{\calI}/\Pi(S^1)$.
\end{definition}

\begin{theorem}\label{thm: one-to-one}
    The set of curves in $\mathdiam{\calQ}$ is in one-to-one correspondence {with} $\calM^0(S^1)$.
    \begin{proof}
    \cite[Theorem 3.2]{charon2020lengthmeasuresplanarclosed} proves the results in the setting of Lipschitz immersions. We may equivalently state the result with absolutely continuous parametrizations because every such convex loop is non-intersective and rectifiable, so it admits a unique arc-length parametrization, which is Lipschitz, so the two quotient spaces coincide.
    \end{proof}
\end{theorem}

A key property of the velocity transformation in the setting of convex loops $\mathdiam{\mathcal{I}}$ is that the image of the Gauss map $c\mapsto T_c$ (or the square root velocity transformation $\Phi$) lies in the subspace of $L^2$ functions whose directions rotate monotonically.
This structural feature is crucial for enabling the optimal transport reformulation presented in Theorem~\ref{thm: elastic-WFR}.

\begin{definition}\label{def: deg-pair}
    We say that a loop $c\in \mathdiam{\calI}$ is degenerate if there exists $r>0$ and $x\in S^1$ such that the associated length measure $\mu_c$ is given by $\mu=r\delta_{x}+r\delta_{-x}$. If this is not the case, we say that $c$ is non-degenerate.
\end{definition}
Intuitively $c$ is a degenerate loop when its image $\operatorname{Im}(c)$ in $\bC$ lies on a segment, so the convex shape that it encloses is not 2-dimensional. This will be a degenerate scenario also in the unbalanced optimal transport framework when complete destruction or creation of mass happens, see Section \ref{sec: adms}.

\subsection{Wasserstein-Fisher-Rao}
The Wasserstein–Fisher–Rao (WFR) metric can be interpreted as a 2-Wasserstein distance on the positive cone, where couplings are required to satisfy a homogeneous marginal constraint, which is a constraint on the second moment, rather than a projection constraint in the standard balanced sense, see Definition \ref{def: wfr}. In the case of the circle $S^1$, this cone is isomorphic to $\mathbb{R}^2\cong \bC$ equipped with the Euclidean distance. Therefore, in the following discussion, we work directly in $\mathbb{C}$ without explicitly constructing the cone. For further details on the general construction, we refer to \cite[Section 7]{Liero_2017}.

Since $S^1$ is compact, the weak* topology on $\calM(S^1)$ coincides with the narrow topology, and we will use the term weak topology unambiguously. This is not the case for the space of measures on the cone, where we will specify depending on the claim.

\begin{definition}\label{def: wfr}
 Given two Borel measures $\mu_0,\mu_1\in \calM(S^1)\subset \calM(\bC)$ (inclusion lift), we define the set of homogeneous couplings as
    \begin{equation}
        \Gamma(\mu_0,\mu_1):=\{\gamma\in \calM_2(\bC\times \bC): (\operatorname{\theta}_i)_\#(\operatorname{r}_i^2\gamma)=\mu_i, \; i=0,1\},
    \end{equation}
    where $\operatorname{r}(x):=|x|$ and $\theta(x)=\frac{x}{|x|}$ on $\bC\setminus\{0\}$ and $\theta(0)=\overline{x}$ for some fixed $\overline{x}\in S^1$. Notice that $\operatorname{r}^2\gamma$ does not charge $0$, so the marginal conditions do not depend on the choice of $\overline{x}$.
    We define the Fisher-Rao energy of a coupling $\gamma\in \Gamma(\mu_0,\mu_1)$ as 
    \begin{equation}
        J(\gamma):=\int_{\bC\times \bC} \left|x- y\right|^2\;\dd\gamma(x,y),
    \end{equation}
    and finally the Wasserstein-Fisher-Rao distance as
    \begin{equation}
        \mathcal{U}(\mu_0,\mu_1):=\inf_{\gamma\in \Gamma(\mu_0,\mu_1)} \sqrt{J(\gamma)}.
    \end{equation}
    If $\Gamma$ is empty, then $\mathcal{U}(\mu_0,\mu_1):=+\infty$.
\end{definition}

Let us start listing some properties of the Wasserstein-Fisher-Rao distance from its convexity.

\begin{theorem}\label{thm: relaxation}
    The squared distance $\mathcal{U}^2$ is the closed convex envelope of the functional $\mathcal{S}:\M(S^{1})\times \M(S^{1})\to \R\cup\{\infty\}$ defined by 
    \begin{equation}
            \mathcal{S}(\mu_0,\mu_1):=\begin{cases}
                |\sqrt{r_0}x_0-\sqrt{r_1}x_1|^2 & \text{ if } \mu_i=r_i\delta_{x_i}\in \Delta_+(S^{1})\\
                +\infty & \text{ otherwise}.
            \end{cases}
        \end{equation}
    \begin{proof}
        From \cite[Theorem 3.5]{SavSod23}.
    \end{proof}
\end{theorem}
\begin{definition}\label{def: weak-adms}
    We say that two measures $\mu_0,\mu_1$ are weakly admissible, if \[\mu_i(\{x\in S^1: \operatorname{dist}(x, \supp(\mu_j))\geq \pi/2\})=0\quad \text{for }i\neq j.\]
\end{definition}
We list now some properties of the optimal couplings in the definition of $\mathcal{U}$.
\begin{proposition}\label{prop: basicWFR}
    Let $\mu_0,\mu_1\in \M_+(S^1)$. It holds
    \begin{enumerate}
        \item\label{it: basicbound} There exists an optimal coupling $\gamma^*\in \Gamma(\mu_0,\mu_1)\cap \calP(\bC\times \bC)$ concentrated on 
        \begin{equation}
            \mathfrak B_R\,\cup\,
            (B^0_R\times\{0\})\cup(\{0\}\times B^0_R),
        \end{equation}
        where $R^2=|\mu_0|(S^1)+|\mu_1|(S^1)$, $B_R:=\{z\in \bC: 0\leq |z|\leq R \}$, $B^0_R:=B_R\setminus \{0\}$, and
    \begin{equation}
        \mathfrak B_R
            :=
            \Bigl\{
                (z,w)\in B^0_R\times B^0_R : 
                \operatorname{dist}\bigl(\theta(z),\theta(w)\bigr)\le \tfrac{\pi}{2}
            \Bigr\}.
    \end{equation}
        \item\label{it: basic3} If $\mu_0,\mu_1$ are weakly admissible, then there exists an optimal coupling $\gamma^*\in \Gamma(\mu_0,\mu_1)\cap \calP(\bC\times \bC)$ concentrated on $\mathfrak B_R$.
        \item\label{it: nar-con} The map $\calP_2(B_R\times B_R)\to \M_+(S^1)$ given by $\gamma\mapsto(\theta_i)_\#(\operatorname{r}_i^2\gamma)$ is narrowly continuous.
    \end{enumerate}
    \begin{proof}
        \cite[Theorem 7.6 + Theorem 7.20]{Liero_2017}. For point \ref{it: nar-con}, the maps $\operatorname{r}_i$ are continuous and bounded in $B_R\times B_R$, making the marginal map narrowly continuous.
    \end{proof}
\end{proposition}
\begin{example}
    Point \ref{it: nar-con} would not be true without the restriction to $B_R\times B_R$. Indeed, the map $\calP_2(\bC\times\bC)\to \M_+(S^1)$ given by $\gamma\mapsto(\theta_i)_\#(\operatorname{r}_i^2\gamma)$ is not narrowly continuous. As a counterexample, let $x,y\in S^1$ and define $\gamma_n=\left(1-\frac{1}{n^2}\right)\delta_{(x, 0)}+ \frac{1}{n^2}\delta_{(x, ny)}$. Then, $\gamma_n\weak \delta_{(x,0)}$ but $(\theta_1)_\#(\operatorname{r}_1^2\gamma_n)=\delta_y$ for all $n$, and $(\theta_1)_\#(\operatorname{r}_1^2\delta_{(x,0)})=0$. An alternative condition would be to ask for uniform bounds on the second moment for the sequence $\gamma_n$.
\end{example}

\subsubsection{Logarithmic entropy transport viewpoint}
We will make use of an alternative formulation of the Wasserstein-Fisher-Rao distance, see \cite[Section 6]{Liero_2017} for more details.
\begin{definition}
    Let $\ell:\R_+\to [0,+\infty]$ be the loss function defined by 
    \begin{equation}
        \ell(\zeta):=\begin{cases}
            -\log(\cos^2(\zeta)) &\zeta\in [0, \pi/2)\\
            +\infty & \zeta\geq \pi/2.
        \end{cases}
    \end{equation}
    Given $\mu_0,\mu_1\in \M_+(S^1)$, we define the forward and reverse logarithmic entropic 
    problems as
    \begin{equation}\label{eq: LETf}
        \calL_F(\mu_0,\mu_1):= \inf\left\{\calF(\mu_0,\mu_1|\eta): \eta\in \calM_+(S^1\times S^1), \eta_i=\sigma_i\mu_i\right\},
    \end{equation}
    \begin{equation}\label{eq: LET}
        \calL_R(\mu_0,\mu_1):= \inf\left\{\calR(\mu_0,\mu_1|\eta): \eta\in \calM_+(S^1\times S^1), \mu_i=\rho_i\eta_i+\mu_i^\perp\right\},
    \end{equation}
    where $\eta_i:=(\operatorname{proj}_i)_\#\eta$, the constraints are in terms of the Lebesgue decomposition of the marginals $\eta_i$ with respect to $\mu_i$, and the energies are given by
    \begin{equation}
        \calF(\mu_0,\mu_1|\eta):=\sum_{i=0}^1\left[\int_{S^1}(\sigma_i\log(\sigma_i)-\sigma_i+1)\dd \mu_i\right] + \int_{S^1\times S^1} \ell(\operatorname{dist(x,y)})\dd\eta(x,y),
    \end{equation}
    \begin{equation}
        \calR(\mu_0,\mu_1|\eta):=\sum_{i=0}^1\left[\mu_i^\perp(S^1) +\int_{S^1}(\rho_i-\log(\rho_i)-1)\dd \eta_i\right] + \int_{S^1\times S^1} \ell(\operatorname{dist(x,y)})\dd\eta(x,y).
    \end{equation}
\end{definition}
 The forward formulation assumes that the marginals $\eta_i$ are absolutely continuous with respect to $\mu_i$, in notation $\eta_i\ll \mu_i$. For such a $\eta$, the densities $\rho_i\in L^1_+(S^1,\eta_i)$ and $\sigma_i\in L^1_+(S^1, \mu_i)$ are related by $\sigma_i\rho_i=1$ whenever they are strictly positive, and are uniquely determined $\eta_i$ almost everywhere \cite[Lemma 2.3]{Liero_2017}.  Moreover, for every Borel subset $I\subset S^1$, it holds
\begin{equation}
    \int_{I} \sigma_i\dd \mu_i^\perp=0.
\end{equation}

\begin{proposition}\label{prop: ex log}
    Let $\mu_0,\mu_1\in \calM_+(S^1)$. It holds
    \begin{enumerate}
        \item\label{it: ex-log} There exists an optimal plan $\eta$ for both $\calL_F$, $\calL_R$, and $\calU^2=\calL_F=\calL_R\leq |\mu_0|(S^1)+|\mu_1|(S^1)$.
        \item\label{it: Kant} Every optimal $\eta^*$ is a solution of the Kantorovich problem between its marginals, i.e.
        \begin{equation}
            \eta^*\in \argmin_{\eta\in \Gamma_b(\eta_0,\eta_1)}\int_{S^1\times S^1} \ell(\operatorname{dist}(x,y))\dd\eta.
        \end{equation}
        \item\label{it: swap} Given an optimal $\eta$ with decompositions $\mu_i=\rho_i\eta_i+\mu_i^\perp$, then there exists an optimal coupling $\gamma$ for $\calU$ such that 
        \begin{equation}
            \gamma=(\sqrt{\rho_0} x, \sqrt{\rho_1}y)_\#\eta +\gamma^\perp,
        \end{equation}
        where $\gamma^\perp\in \Gamma(\mu_0^\perp, \mu_1^\perp)$ is supported on $(\bC\times \{0\})\cup(\{0\}\times \bC)$.
        \item\label{it:annih1} Any optimal $\eta$ with decompositions $\mu_i=\rho_i\eta_i+\mu_i^\perp$, also satisfies
        \begin{equation}
            \mu_i^\perp = \mu_i\mres \{x\in S^1: \operatorname{dist}(x, \supp(\mu_j))\}\geq \pi/2\}. 
        \end{equation}
        \item\label{it: annih2} Given $\mu_i^\perp := \mu_i\mres \{x\in S^1: \operatorname{dist}(x, \supp(\mu_j))\}\geq \pi/2\}, $ it holds
        \begin{equation}
            \calL_R(\mu_0, \mu_1)=\calL_R(\mu_0-\mu^\perp_0, \mu_1-\mu^\perp_1)+|\mu_0^\perp|(S^1)+|\mu_1^\perp|(S^1).
        \end{equation}
    \end{enumerate}
    \begin{proof}
        The existence of a solution in point \ref{it: ex-log} that attains the minimum for both the forward and backward problems is given by \cite[Theorem 6.2b]{Liero_2017}. The equivalence with $\calU$ is proved in \cite[Theorem 7.20]{Liero_2017}. Point \ref{it: Kant} is \cite[Theorem 6.3c]{Liero_2017}. Point \ref{it: swap} follows from \cite[Theorem 7.20(iii)]{Liero_2017} and the construction of the relaxed problem \cite[Lemma 7.9]{Liero_2017}. Point \ref{it:annih1} is \cite[Point a) after Theorem 7.20]{Liero_2017}. Point \ref{it: annih2} is a consequence of the previous point.
    \end{proof}
\end{proposition}

\begin{corollary}\label{cor: weak-adms}
    Let $\mu_0,\mu_1\in \calM_+(S^1)$ be weakly admissible measures. Then there exists an optimal plan $\eta$ for which $\sigma_i:=d\eta_i/d\mu_i\in L^1_+(S^1,\mu_i)$ are strictly positive $\mu_i$-a.e. $x\in S^1$.
    \begin{proof}
        Since the measures are weakly admissible, we can expect $\mu_i^\perp=0$ on point \ref{it:annih1}. Then, $\rho_i\in L^1_+(S^1,\mu_i)$ is strictly positive $\mu_i$ a.e. and so is $\sigma_i$.
    \end{proof}
\end{corollary}
To get additional regularity properties of the densities $\rho_i, \sigma_i$ as proved in \cite{gallouet2024regularitytheorygeometryunbalanced}, we need a stronger admissibility condition on the measures $\mu_0,\mu_1$.
\begin{definition}\label{def: adms}
    Two measures $\mu_0,\mu_1\in \calM_+(S^1)$ are admissible if they satisfy an $\ell$-Hausdorff condition, i.e.
    \begin{equation}\label{eq: ell-Hau}
            \max\left(\sup_{x\in\supp(\mu_0)}\inf_{y\in \supp(\mu_1)} \ell(\operatorname{dist(x,y)}), \sup_{y\in\supp(\mu_1)}\inf_{x\in \supp(\mu_0)} \ell(\operatorname{dist(x,y)})\right)<\infty.
        \end{equation}
\end{definition}
\begin{lemma}\label{lemma: adms2}
    Let $\mu_0,\mu_1\in \calM_+(S^1)$ be admissible measures. Then, they are weakly admissible.
    \begin{proof}
        The proof is straightforward from the definition of $\ell$.
    \end{proof}
\end{lemma}
The opposite implication is not necessarily true. In the case of vanishing first moment measures, it only fails when one of the two measures is a dipole, as we will see in Section \ref{sec: adms}. 

\begin{proposition}\label{prop: annih}
    Let $\mu_0,\mu_1\in \calM_+(S^1)$ be admissible measures. Then there exists an optimal plan $\eta$ for which $\sigma_i:=d\eta_i/d\mu_i\in L^\infty_+(S^1,\mu_i)$ are Lipschitz continuous on the support of $\mu_i$ and $\inf_{s\in \supp\mu_i}\sigma_i(s)>0$.
    \begin{proof}
        \cite[Corollary 8]{gallouet2024regularitytheorygeometryunbalanced} ensures th{e} existence (and uniqueness) of Lipschitz continuous potentials $(u_0,u_1)$. By \cite[Lemma 3]{gallouet2024regularitytheorygeometryunbalanced}, these Lipschitz continuous potentials are optimal for a standard optimal transport problem between $\eta_i=F_i^*(-u_i)\mu_i$, where $F_i$ is the considered entropy and $F_i^*$ is its Legendre-Fenchel transform. In our case, $F_i(x)=x\log(x)-x+1$, and $F_i^*(x)=e^x-1$. Thus, $\sigma_i=e^{-u_i}$ satisfies the properties in the claim, by the Lipschitz continuity of $u_i$, and the compactness of $S^1$.
    \end{proof}
\end{proposition}

\section{The SRVT distance is Wasserstein-Fisher-Rao}\label{sec: equi}
In this section, we prove the equivalence between the square root velocity distance between convex loops and the Wasserstein–Fisher–Rao metric (Theorem \ref{thm: elastic-WFR}). 
To do so, we need three main ingredients. First, we study the admissibility properties of the length measures in order to use the regularity result in Proposition \ref{prop: annih} and Corollary \ref{cor: weak-adms}.
Second, we lift both the curves and their length measures to the universal cover of $S^1$, which allows us to work with periodic densities on $\mathbb R$. In turn, we characterize admissible homogeneous couplings in terms of these lifts and show that optimal couplings satisfy a monotonicity property.
Third, we use this characterization to identify optimal reparametrizations for the SRVT distance and to establish the equivalence with the WFR formulation.

The main idea is the following: let $c_0,c_1\in \mathdiam{\calI}$ be two convex loops, and $\mu_0,\mu_1$ be their length measures. Based on the results in Proposition \ref{prop: ex log}, specifically point \ref{it: Kant}, we can find minimizers of the Wasserstein-Fisher-Rao energy between $\mu_0$ and $\mu_1$, by minimizing the following functional over densities $\sigma_i\in L^1_+(S^1, \mu_i)$ and a Kantorovich problem with marginals $\sigma_i\mu_i$:
\begin{equation}\label{eq: ent-non-deg}
    \overline{\calF}(\sigma_0,\sigma_1):=\sum_{i=0}^1\int_{S^1}\left[1 + \sigma_i\log(\sigma_i)-\sigma_i\right]\dd\mu_i +\inf_{\eta\in \Gamma_b\left(\sigma_0\mu_0, \sigma_1\mu_1\right)}\int_{S^\times S^1} \ell(\operatorname{dist}(x,y))\dd\eta(x,y).
\end{equation}
The existence of minimizers of $\calF$ in $L_+^1(S^1, \mu_i)$ is guaranteed by Proposition \ref{prop: ex log}. When the measures $\mu_i$ are admissible or weakly admissible, minimizers can be taken with $\inf_{s\in \supp(\mu_i)}\sigma_i(s)>0$ by Proposition \ref{prop: annih} or strictly positive $\mu_i$-a.e. $x\in S^1$ by Corollary \ref{cor: weak-adms}. We then study in detail the Kantorovich term in Equation \eqref{eq: ent-non-deg} under such positivity assumptions, using both a pullback construction through the Gauss map and a lifting procedure from $S^1$ to $\R$. This leads to explicit Monge maps obtained by composing cumulative distribution functions and their inverses, that are very much related to an optimal reparametrization for the SRVT distance. Even though the length measures $\mu_i$ could be singular (e.g. if the loops have corners), the pullback construction ensures that the measures used in this last step are absolutely continuous and supported everywhere, so that the mentioned inversion is possible.

\subsection{Admissibility of length measures}\label{sec: adms}
This section is focused on studying when the length measures are admissible or weakly admissible, so that Proposition \ref{prop: annih} or Corollary \ref{cor: weak-adms} applies to them.
\begin{lemma}\label{lemma: non-deg, adm}
    Let $c_0, c_1\in \mathdiam{\calI}$ be non-degenerate convex loops. Their length measures
    $\mu_0,\mu_1\in \calM^0(S^1)$ are admissible.
        \begin{proof}
            Assume by contradiction that 
            \begin{equation}
                \sup_{x\in \supp(\mu_0)}\inf_{y\in \supp(\mu_1)}\ell(\operatorname{dist}(x,y))=\infty.
            \end{equation}
            By compactness of $\supp\mu_0$ and definition of $\ell$, there exists $x\in \supp(\mu_0)$ such that 
            \begin{equation}
                \operatorname{dist}(x, \supp(\mu_1))=\inf_{y\in \supp(\mu_1)}\operatorname{dist}(x,y)\geq \pi/2.
            \end{equation}
            It means that the support $\operatorname{supp}(\mu_1)$ is contained in the closed semicircle
        \begin{equation}
            H_x := \{\, y\in S^1:\ x\cdot y \le 0 \,\}.
        \end{equation}
        In particular, $\mu_1(S^1\setminus H_x)=0.$
        We now use the vanishing first moment condition $\mu_1\in\calM^0(S^1)$ and the property of coordinate-wise integration:
        \begin{equation}
            0 = x\cdot \int_{S^1} y\,\dd\mu_1(y) = \int_{S^1} (x\cdot y)\,\dd\mu_1(y).
        \end{equation}
        On $H_x$ we have $x\cdot y\le 0$, and on $S^1\setminus H_x$ we have $x\cdot y>0$.
        Together with $\mu_1(S^1\setminus H_x)=0$, this yields
        \begin{equation}
            0=\int_{H_x} (x\cdot y)\,\dd\mu_1(y)\le 0.
        \end{equation}
        Therefore equality forces $x\cdot y=0$ for $\mu_1$-a.e.\ $y$, i.e.
        \begin{equation}\label{eq: support_on_equator}
            \operatorname{supp}(\mu_1)\subseteq \{\,y\in S^1:\ x\cdot y=0\,\}.
        \end{equation}
        The right-hand side consists of exactly two antipodal points, which is in contradiction with the assumption that $c_1$ is non-degenerate. The other side of the condition is analogous, using the non degeneracy of $c_0$.
        \end{proof}
\end{lemma}

Admissibility is not automatic when one of the two curves is degenerate:
\begin{lemma}\label{lemma: deg-adms}
    Let $c_0, c_1\in \mathdiam{\calI}$ be two convex loops, and assume that $c_j$ is degenerate of the type $\mu_j=r\delta_{x_j}+r\delta_{-x_j}$. Denote by $\{\pm\;\i x_j\}$ the two antipodal points in $S^1$ orthogonal to $x_j$. 
    Then, $\mu_0,\mu_1$ are admissible if and only if $\{\pm\;\i x_j\}\cap \supp(\mu_i)=\emptyset$.
    \begin{proof}
    Let's first look at 
    \begin{equation}\label{eq: lH1}
        \sup_{x\in \supp(\mu_j)}\inf_{y\in \supp(\mu_i)}\ell(\operatorname{dist}(x,y))=\max\left(\inf_{y\in \supp(\mu_i)}\ell(\operatorname{dist}(x_j,y), \inf_{y\in \supp(\mu_i)}\ell(\operatorname{dist}(-x_j,y)\right).
    \end{equation}
    It is finite if {and} only if $\supp(\mu_i)\cap \mathring{H}_{x_j}\neq \emptyset$ and $\supp(\mu_i)\cap \mathring{H}_{-x_j}\neq \emptyset$, where $\mathring{H}_x=\{y\in S^1: x\cdot y<0\}$. Since $\mu_i\in \calM^0(S^1)$,  $\supp(\mu_i)\cap \mathring{H}_{x_j}\neq \emptyset$ implies $\supp(\mu_i)\cap \mathring{H}_{-x_j}\neq \emptyset$, and vice versa. Therefore, \eqref{eq: lH1} is finite if and only if $\supp(\mu_i)\not\subseteq \{\pm\;\i x_j\}$, which is equivalent to say that $c_i$ is not degenerate in the direction $\i x_j$. Regarding the second Hausdorff term:
    \begin{equation}\label{eq: lH2}
        \sup_{y\in \supp(\mu_i)}\inf_{x\in \supp(\mu_j)}\ell(\operatorname{dist}(x,y))=\sup_{y\in \supp(\mu_i)}\min\left(\ell(\operatorname{dist}(x_j,y), \ell(\operatorname{dist}(-x_j,y)\right).
    \end{equation}
    For both points $\{\pm\;\i x_j\}$, the two terms in the minimization are infinite, but this is the only case.
    Therefore, \eqref{eq: lH2} is finite if {and} only if $\{\pm\;\i x_j\}\cap \supp(\mu_i)=\emptyset$.
    \end{proof}
\end{lemma}
\begin{corollary}\label{cor: deg-adms}
    Let $c_0, c_1\in \mathdiam{\calI}$ be two convex loops, and assume that $c_j$ is degenerate of the type $\mu_j=r\delta_{x_j}+r\delta_{-x_j}$. 
    Then, 
    the length measures $\mu_0,\mu_1$ are weakly admissible if and only if $\mu_i(\{\pm\;\i x_j\})=0$.
    \begin{proof}
        The proof is straightforward and follows similar arguments to Lemma \ref{lemma: non-deg, adm}.
    \end{proof}
\end{corollary}

\begin{example}\label{ex: seg-square}
    Let $c_0$ be a degenerate loop of the form $\mu_0=3\delta_0+3\delta_{\pi}$ and $c_1$ be a square with length measure $\mu_1=\delta_0+\delta_{\pi/2}+\delta_\pi+ \delta_{3\pi/2}$. By the previous discussion, we can expect minimizers of Equation \ref{eq: ent-non-deg} to be
    \begin{equation}
        \sigma_0(s)=\begin{cases}
            3^{-3/4} &s=0,\pi\\
            0& \text{else}
        \end{cases}\qquad \sigma_1(s)=\begin{cases}
            3^{1/4} &s=0,\pi\\
            0& \text{else}
        \end{cases}.
    \end{equation}
    The Kantorovich term in Equation \ref{eq: ent-non-deg} vanishes in this case, and the entropic part splits into the cost $|\mu_1^\perp|(S^1)=2$ of complete creation of mass at $\sigma_1(\pi/2)=\sigma_1(3\pi/2)=0$ and the cost $\approx 0.24$ of removing mass in the directions $0,\pi$. Rotating the direction of degeneracy of $c_0$ slightly, which means having $\mu_0^k=3\delta_{1/k}+3\delta_{\pi+1/k}$ for $k$ large, changes the situation entirely as $\mu_0^k$, $\mu_1$ become (weakly) admissible for all $k$ and Proposition \ref{prop: annih} holds. Since $\mu_0^k$ converges to $\mu_0$ and $\calU$ metrizes the weak topology, we get $\calU(\mu_0^k, \mu_0)\to 0$.
\end{example}

\begin{lemma}\label{lemma: seq-weak-adms}
    Let $c_0, c_1\in \mathdiam{\calI}$ be two convex loops, and assume that $c_j$ is degenerate of the type $\mu_j=r\delta_{x_j}+r\delta_{-x_j}$. There exists a sequence $x_j^k\to x_j$ for $k\to \infty$ defining $\mu_j^{k}:=r\delta_{x_j^k}+r\delta_{-x_j^k}$ such that $\mu_j^{k},\mu_i$ are weakly admissible for all $k$.
    \begin{proof}
        This is a consequence of Corollary \ref{cor: deg-adms}, and the fact that $\mu_i$ can have at most countably many atoms with positive mass. So, we can just use any sequence $x_j^k$ that avoids those points.
    \end{proof}
\end{lemma}
In the setting of the previous {l}emma, {i}f $c_1$ is also polygonal, it holds that $\mu_0^k$, $\mu_1$ are admissible for all $k$. 

\subsection{Lift to the universal cover and one-dimensional optimal transport}\label{sec: lift}

In this section, we construct a framework for lifting measures on $S^1$ to the
universal cover $\mathbb R$.
This lifting procedure allows us to replace transport problems on the circle
by one-dimensional transport problems on the real line, for which standard
monotonicity results for optimal transport maps are available (see for example \cite[Sec.~2.1]{Santa}). Throughout this section, convexity plays a crucial role in ensuring the equivalence between the optimal transport problems associated with length
measures and arc-length measures, and, ultimately, in guaranteeing the
existence of optimal reparametrizations for non-degenerate loops.

Let $q:\R\to \R/2\pi\Z\cong S^1$ be the canonical quotient map to the circle.  
Consider the space of $2\pi$-periodic curve lifts
\begin{equation}
    \calP := \big\{\narrowcheck{c}=c\circ q : c\in \calI\big\},
    \qquad 
    \mathdiam{\calP} := \big\{\narrowcheck{c}=c\circ q : c\in \mathdiam{\calI}\big\}.
\end{equation}
Each $\narrowcheck{c}\in \calP$ represents a periodic version of a curve $c\in \calI$.
We also introduce the lifted space of reparametrizations:
\begin{equation}
    \Pi(\R)
    := \Big\{
        \phi\in W^{1,1}_{\loc}(\R):
        \ \phi'(t)>0 ,\ 
        \phi(t)+2\pi=\phi(t+2\pi) \; \text{a.e. on }\R
    \Big\}.
\end{equation}
Elements of $\Pi(\R)$ correspond to orientation-preserving reparametrizations of $\R$ that project to $\Pi(S^1)$ through $q$.

\begin{lemma}\label{lemma: prop-lift}
    Let $c\in \calI$ and $\phi\in \Pi(S^1)$, and choose lifts $\narrowcheck{c}\in \calP$ and $\narrowcheck{\phi}\in \Pi(\R)$. Then:
    \begin{enumerate}
        \item $\narrowcheck{c}\circ \narrowcheck{\phi}$ is a lift of $c\circ \phi$;
        \item $\phi\circ q = q\circ \narrowcheck{\phi}$.
    \end{enumerate}
\end{lemma}

In the case of convex loops, the Gauss map admits a well-behaved lift to $\R$.
This property is essential for constructing monotone transport maps between length measures.

\begin{lemma}\label{lemma: convexity}
    Let $c\in \mathdiam{\calI}$ be a convex loop, and let $T_c = c'/|c'|$ denote its Gauss map.  
    Then there exists a unique left continuous and non-decreasing function $\narrowcheck{T}_c:\R\to \R$ such that
    \begin{equation}
        q\circ \narrowcheck{T}_c = T_c\circ q, 
        \qquad 
        \narrowcheck{T}_c(0)\in [0,2\pi),
    \end{equation}
    and $\narrowcheck{T}_c$ has jumps of size at most $\pi$. Moreover, for any $\alpha\in \R$ we have
    \begin{equation}
        q_\#\narrowcheck{\mu}_c^\alpha = \mu_c,
        \qquad 
        \text{where}\qquad
        \narrowcheck{\mu}_c^\alpha :=
        (\narrowcheck{T}_c)_\#(|\narrowcheck{c}'|\dd t\mres [\alpha,\alpha+2\pi)).
    \end{equation}

    \begin{proof}
        The convexity of $c$ and the orientation assumption ensure that its Gauss map $T_c$ is counterclockwise rotationally monotone, which allows the construction of a non-decreasing lift $\narrowcheck{T}_c$ with jumps of size at most $\pi$.  

        For the second claim, let $\narrowcheck{c}$ be a lift of $c$. By construction,
        \begin{equation}
            q_\#(|\narrowcheck{c}'|\dd t\mres [\alpha,\alpha+2\pi))
            = |c'|\dd s.
        \end{equation}
        Combining this with $q\circ \narrowcheck{T}_c=T_c\circ q$ gives
        \begin{equation}
            (q\circ \narrowcheck{T}_c)_\#(|\narrowcheck{c}'|\dd t\mres [\alpha,\alpha+2\pi))
            = (T_c)_\#(|c'|\dd s)
            = \mu_c,
        \end{equation}
        which completes the proof.
    \end{proof}
\end{lemma}

\begin{lemma}\label{lemma: resc}
    Let $\phi\in \Pi(S^1)$ and $c\in \mathdiam{\calI}$. For all $\alpha\in \R$, there is some $\beta\in \R$ such that $\narrowcheck{\mu}_c^\alpha = \narrowcheck{\mu}_{c\circ \phi}^\beta$.

\begin{proof}  
    On the circle, we have $\mu_c=\mu_{c\circ \phi}$; however, their lifts to $\R$ require more care.  
    Let $\narrowcheck{\mu}_c^\alpha=(\narrowcheck{T}_c)_\#(|\narrowcheck{c}'|\dd t\mres [\alpha,\alpha+2\pi))$ be a lift of $\mu_c$ as in Lemma~\ref{lemma: convexity}, and let $\narrowcheck{\phi}\in \Pi(\R)$ be a lift of $\phi$.  
    Using part~(2) of Lemma~\ref{lemma: prop-lift}, we have
    \begin{equation}
        q\circ \narrowcheck{T}_{c\circ \phi}
        = T_{c\circ\phi}\circ q
        = T_c\circ \phi\circ q
        = T_c\circ q\circ \narrowcheck{\phi}
        = q\circ \narrowcheck{T}_c\circ \narrowcheck{\phi}.
    \end{equation}
    Since $\narrowcheck{T}_c$ and $\narrowcheck{T}_{c\circ\phi}$ have jumps of size at most $\pi$ and $\narrowcheck{\phi}$ is continuous, we conclude that
    \begin{equation}
        \narrowcheck{T}_{c\circ\phi} = \narrowcheck{T}_c\circ \narrowcheck{\phi}.
    \end{equation}
    Setting $\beta=\narrowcheck{\phi}(\alpha)$, it follows that the lifted measures satisfy $\narrowcheck{\mu}_c^\alpha = \narrowcheck{\mu}_{c\circ \phi}^\beta$.
\end{proof}
\end{lemma}

\begin{lemma}\label{lemma: rhomu}
Let $c\in \mathdiam{\calI}$ and let $\sigma\in L^1_+(S^1, \mu_c)$.   
    Then $\narrowcheck{\psi}:=\narrowcheck{\sigma}\circ \narrowcheck{T}_c\in L^1_{\loc}(\R)$ is $2\pi$-periodic and a lift of $\psi:=\sigma\circ T_c$, and for every $\alpha\in \R$, it holds
    \begin{equation}
        q_\#\narrowcheck{\eta}_c^\alpha = \sigma\,\mu_c,
        \qquad
        \text{where}\qquad
        \narrowcheck{\eta}_c^\alpha := (\narrowcheck{T}_c)_\#(\narrowcheck{\nu}_c^\alpha), \qquad \narrowcheck{\nu}_c^\alpha:=\narrowcheck{\psi}|\narrowcheck{c}'|\dd t\mres [\alpha,\alpha+2\pi)
    \end{equation}

    \begin{proof} 
        By construction,
        \begin{equation}
            q_\#(\narrowcheck{\psi}|\narrowcheck{c}'|\dd t\mres [\alpha,\alpha+2\pi))
            = \psi|c'|\dd s,
        \end{equation}
        and using $q\circ \narrowcheck{T}_c = T_c\circ q$ yields the claim.
    \end{proof}
\end{lemma}

\begin{lemma}\label{lemma: equivalence Gauss}
    Let $c_0,c_1\in \mathdiam{\calI}$ be convex loops with length measures $\mu_0,\mu_1$, and $\sigma_i\in L_+^1(S^1, \mu_i)$ be two Borel functions defining probability measures $\sigma_i\mu_i$. For any coupling $\eta\in \Gamma_b(\sigma_0\mu_0,\sigma_1\mu_1)$, there exists a coupling $\nu\in \Gamma_b(\psi_0|c'_0|\dd s, \psi_1|c'_1|\dd s)$ such that 
    \begin{equation}
        \eta=(T_{c_0}\times T_{c_1})_\#\nu.
    \end{equation}
    \begin{proof}
        For $\alpha_i\in \R$, consider the measures $\narrowcheck{\eta}_i^{\alpha_i}$ and $\narrowcheck{\nu}_i^{\alpha_i}$ of Lemma \ref{lemma: rhomu}. We first prove that, for any coupling $\narrowcheck{\eta}\in \Gamma_b(\narrowcheck{\eta}_0^{\alpha_0},\narrowcheck{\eta}_1^{\alpha_1})$, there exists a coupling $\narrowcheck{\nu}\in \Gamma_b(\narrowcheck{\nu}_0^{\alpha_0}, \narrowcheck{\nu}_1^{\alpha_1})$ such that 
    \begin{equation}
        \narrowcheck{\eta}=(\narrowcheck{T}_{c_0}\times \narrowcheck{T}_{c_1})_\#\narrowcheck{\nu}.
    \end{equation}
        Note that $\narrowcheck{\eta}_i^{\alpha_i}$ can only have atoms when there is an interval on which $\narrowcheck{T}_{c_i}$ is constant. So if both $\narrowcheck{\eta}_i^{\alpha_i}$ are atomless, then $\narrowcheck{T}_{c_i}$ are strictly increasing, hence injective, and we can define $\narrowcheck{\nu}(A\times B):=\narrowcheck{\eta}(\narrowcheck{T}_{c_0}(A)\times \narrowcheck{T}_{c_1}(B))$ for $A,B \subset \R$ arbitrary Borel sets, a premeasure on the algebra generated by rectangles, which extends uniquely to a Borel measure on $\R^2$ by the Hahn-Kolmogorov extension theorem.
        Assume that $\narrowcheck{\eta}_0^{\alpha_0}$ has a countable set of atoms $\{t_i\}_{i\in \N}$ with masses $\{m_i\}_{i\in\N}$. We define the conditional probability measure on $\R$
        \begin{equation}
            \gamma(\cdot|t_i):=\frac{1}{m_i}\narrowcheck{\eta}(\{t_i\}\times \cdot),
        \end{equation}
        and the residue measure
        \begin{equation}
            \overline{\gamma}(A\times B):=\narrowcheck{\eta}(A\times B)-\sum_{i}m_i\delta_{t_i}(A)\gamma(B|t_i).
        \end{equation}
        Finally, we define the transport plan $\kappa$ between $\narrowcheck{\nu}_0^{\alpha_0}$ and $\narrowcheck{\eta}_1^{\alpha_1}$ using the same extension procedure on
        \begin{equation}
            \kappa(C\times B):=\sum_{i}\narrowcheck{\nu}_0^{\alpha_0}(C\cap \narrowcheck{T}_{c_0}^{-1}(\{t_i\}))\gamma(B|t_i) + \overline{\gamma}(\narrowcheck{T}_{c_0}(C)\times B).
        \end{equation}
        By construction, $\kappa(\narrowcheck{T}_{c_0}^{-1}(A)\times B)= \narrowcheck{\eta}(A\times B)$. Indeed, 
        \begin{equation}
            \overline{\gamma}(\narrowcheck{T}_{c_0}(\narrowcheck{T}^{-1}_{c_0}(A))\times B)=\narrowcheck{\eta}(A\times B)-\sum_i m_i\delta_{t_i}(A)\gamma(B|t_i),
        \end{equation}
        where $\narrowcheck{\eta}(\narrowcheck{T}_{c_0}(\narrowcheck{T}^{-1}_{c_0}(A))\times B)=\narrowcheck{\eta}(A\times B)$ because $\eta$ is not supported on $(A\setminus \narrowcheck{T}_{c_0}(\narrowcheck{T}^{-1}_{c_0}(A)))\times B$ by the marginal constraints, and $\delta_{t_i}(A)=\delta_{t_i}(\narrowcheck{T}_{c_0}(\narrowcheck{T}^{-1}_{c_0}(A)))$ for a similar reason since $t_i$ are in the support of $\narrowcheck{\eta}_0^{\alpha_0}$. 
        Together with 
        \begin{equation}
            \sum_{i}\narrowcheck{\nu}_0^{\alpha_0}(\narrowcheck{T}_{c_0}^{-1}(A\cap \{t_i\}))\gamma(B|t_i)=\sum_{i}\narrowcheck{\eta}_0^{\alpha_0}(A\cap \{t_i\})\gamma(B|t_i)= \sum_{i}m_i\delta_{t_i}(A)\gamma(B|t_i), 
        \end{equation}
        we prove the equality $\kappa(\narrowcheck{T}_{c_0}^{-1}(A)\times B)= \narrowcheck{\eta}(A\times B)$. 
        By an analogous construction on the second argument and by uniqueness of the extensions, we obtain the claim.

        Now note that every $\eta\in \Gamma_b(\sigma_0\mu_0,\sigma_1\mu_1)$ is of the type $(q\times q)_\#\narrowcheck{\eta}$ for some $\narrowcheck{\eta}\in \Gamma(\narrowcheck{\eta}_0^{\alpha_0},\narrowcheck{\eta}_1^{\alpha_1})$, which concludes the proof.
    \end{proof}
\end{lemma}

The previous Lemma ensures that the map 
\begin{equation}
    (T_{c_0}\times T_{c_1})_\#:\Gamma_b(\psi_0|c'_0|\dd s, \psi_1|c'_1|\dd s)\to \Gamma_b(\sigma_0\mu_0,\sigma_1\mu_1)
\end{equation}
is surjective. Therefore, Kantorovich problems between the measures $\sigma_0\mu_0$ and $\sigma_1\mu_1$ can be translated in terms of the measures $\nu_0:=\psi_0|c'_0|\dd s$ and $\nu_1:=\psi_1|c'_1|\dd s$.
Consider the Kantorovich part of the entropic formulation in Equation \eqref{eq: ent-non-deg}, and let $c(x,y):=\ell(\operatorname{dist}(x,y))$ the cost function on the circle. One way to lift it to the universal cover is the following: for $t,s\in \R$,
\begin{equation}
    c(q(t), q(s))=\inf_{k\in \Z} \narrowcheck{c}(t,s +2\pi k), \qquad \narrowcheck{c}(t,s):=\ell(|t-s|).
\end{equation}
The infimum is attained since $\narrowcheck{c}$ is invariant to $2\pi k$ shifts (i.e. $\narrowcheck{c}(t+2\pi k, s+2\pi k)=\narrowcheck{c}(t,s)$), and it grows uniformly as $|t-s|\to \infty$.
In turn, we get the equivalences,
\begin{equation}
    c(T_{c_0}(q(t)),T_{c_1}(q(s)))=\inf_{k\in \Z} \narrowcheck{c}(\narrowcheck{T}_{c_0}(t), \narrowcheck{T}_{c_1}(s+2\pi k)).
\end{equation}
Denote by $c_T, \narrowcheck{c}_T$ the compositions $c(T_{c_0}, T_{c_1}), \narrowcheck{c}(\narrowcheck{T}_{c_0}, \narrowcheck{T}_{c_1})$ respectively. The cost function $\narrowcheck{c}_T$ on $\R\times \R$ satisfies the following properties:
\begin{enumerate}
    \item It is lower semicontinuous, because of the lower semicontinuity of $\ell(|t-s|)$ and $\narrowcheck{T}_{c_i}$, which are non-decreasing and left-continuous;
    \item \label{it: growth} It satisfies a growth condition, for any $\lambda\in \R$ there exists $R(\lambda) >0$ such that $\narrowcheck{c}_T(t,s)\geq R(\lambda)$ whenever $|t-s|\geq \lambda$;
    \item It is invariant with respect to $2\pi$ shifts, i.e. $\narrowcheck{c}_T(t+2\pi, s+2\pi)=\narrowcheck{c}_T(t,s)$;
    \item \label{it: monge} It satisfies a (non-strict) Monge condition, i.e. for $s_0\leq s_1$ and $t_0\leq t_1$, it holds
    \begin{equation}
        \narrowcheck{c}_T(t_0,s_0)+\narrowcheck{c}_T(t_1,s_1)\leq \narrowcheck{c}_T(t_0, s_1)+\narrowcheck{c}_T(t_1, s_0).
    \end{equation}
\end{enumerate}
These four properties are in line with the assumptions made in \cite{Delon_2010}, where the authors proved a monotonicity result of transport plans in the circle which is consistent with the well-known monotonicity results in one-dimensional optimal transport \cite[Chapter 2]{Santa}.  
The main differences are that their cost is $\R$-valued, and, in property \ref{it: monge}, where the authors assume a strict inequality to prove that all the minimizers are induced by a monotone optimal plan in the universal cover. Nevertheless, their approach using local modifications of the plans to force monotonicity without increasing the cost is still valid under weaker assumptions, see Appendix \ref{sec: app1}.

\begin{definition}
    Given a $2\pi$ periodic measure $\narrowcheck{\zeta}$ on $\R$ that is normalized to unit mass over each period, we define its cumulative distribution function $F_{\widecheck{\zeta}}$ centered at $0$ as
\begin{equation}
    F_{\widecheck{\zeta}}(t)=\begin{cases}
        \narrowcheck{\zeta}([0, t)) &t\geq 0;\\
        -\narrowcheck{\zeta}((t, 0]) &t<0.
    \end{cases}
\end{equation}
\end{definition}
This function is always non decreasing and right continuous. Whenever $\narrowcheck{\zeta}$ is atomless, $F_{\widecheck{\zeta}}:\R\to \R$ is continuous and surjective. When $\narrowcheck{\zeta}$ is globally supported, then $F_{\widecheck{\zeta}}:\R\to \R$ is injective.
Recall that the pseudo-inverse of $F_{\widecheck{\zeta}}$ is defined by
\begin{equation}
    F_{\widecheck{\zeta}}^{[-1]}(t):=\inf \{s\in \R: F_{\widecheck{\zeta}}(s)> t\}.
\end{equation}
This map is non decreasing and left continuous.
Another property of $F_{\widecheck{\zeta}}^{[-1]}$ is that it transports the Lebesgue measure $\dd t$ to $\narrowcheck{\zeta}$, i.e. $\left(F_{\widecheck{\zeta}}^{[-1]}\right)_\#\dd t=\narrowcheck{\zeta}$. 

\begin{lemma}\label{lemma: Delon}
    Let $c_0,c_1\in \mathdiam{\calI}$ be convex loops with length measures $\mu_0,\mu_1$, and $\sigma_i\in L^1_+(S^1,\mu_i)$ be two functions defining probability measures $\sigma_i\mu_i$. 
    Consider the measures $\nu_i:= \psi_i |c_i'|\dd s$ with $\psi_i:=\sigma_i\circ T_{c_i}$, and the lifts $\narrowcheck{\nu}_i=\narrowcheck{\psi}_i|\narrowcheck{c}_i'|\dd t$. Assume that 
    \begin{equation}\label{eq: inf}
        \inf_{\nu\in \Gamma_b(\nu_0, \nu_1)}\int_{S^1\times S^1} c_T(x,y)\dd \nu<+\infty.
    \end{equation}
    
    There exists $\theta\in \R$ and $\narrowcheck{\nu}_\theta\in \Gamma_b(\narrowcheck{\nu}_0, \narrowcheck{\nu}_1)$ such that 
    \begin{equation}\label{eq: nutheta1}
        \narrowcheck{\nu}_{\theta}(A\times B)=\narrowcheck{\zeta}^\theta(F_{\widecheck{\nu}_0}(A)\times F_{\widecheck{\nu}_1}(B)), \qquad \narrowcheck{\zeta}^\theta:=(\Id\times (\Id+\theta))_\#\dd t,
    \end{equation}
    and projects to a minimizer of \eqref{eq: inf}.
    \begin{proof}
        Let $\nu^*$ be a minimizer of \eqref{eq: inf}, and let $\narrowcheck{\nu}^*\in \Gamma_b(\narrowcheck{\nu}_0,\narrowcheck{\nu}_1)$ be a lift supported on the closure of
        \begin{equation}
              S:=\{(t,s)\in \R\times \R: c_T(q(t), q(s))=\narrowcheck{c}_T(t,s)<\infty\}.
         \end{equation}
         This can be done because \eqref{eq: inf} is finite, so we can choose a representer of $\nu^*$ supported on $\operatorname{cl}S\cap[0,2\pi]^2$ and expand it by periodicity. Then, by construction,
        \begin{equation}
            \int_{S^1\times S^1}c_T(x,y)\dd \nu^*=\int_{[0,2\pi)^2} \narrowcheck{c}_T(t,s)\dd\narrowcheck{\nu}^*<\infty.
        \end{equation}
        Now by \cite[Lemma 4.4]{Delon_2010} there exists $\narrowcheck{\zeta}^*\in \Gamma_b(\dd t,\dd t)$ such that $\narrowcheck{\nu}^*(A\times B)=\narrowcheck{\zeta}^*(F_{\widecheck{\nu}_0}(A)\times F_{\widecheck{\nu}_1}(B))$. 
        By \cite[Lemma 4.14]{Delon_2010}, $\narrowcheck{\zeta}^*$ is cost equivalent to a shift measure $\narrowcheck{\zeta}^\theta:=(\Id\times (\Id+\theta))_\#\dd t$. For completeness, we prove in Appendix \ref{sec: app1} that such a statement is still true under our weaker assumptions.
    \end{proof}
    
\end{lemma}
When $F_{\widecheck{\nu}_0}$ and $F_{\widecheck{\nu}_1}$ are homeomorphisms, we can equivalently say that
\begin{equation}
        \nu^*:=(q\times q)_\#\narrowcheck\nu_\theta, \qquad \narrowcheck{\nu}_{\theta}:=(F_{\widecheck{\nu}_0}^{-1}\times (F_{\widecheck{\nu}_1}-\theta)^{-1})_\#(\dd t\mres [0,1]),
    \end{equation}
is a minimizer of \eqref{eq: inf}. It is the case in the next result.

\begin{proposition}\label{pro: final-eq}
    Let $c_0,c_1\in \mathdiam{\calI}$ be convex loops with length measures $\mu_0,\mu_1$, and $\sigma_i\in L_+^1(S^1,\mu_i)$ be two $\mu_i$-almost everywhere strictly positive functions, defining probability measures $\sigma_i\mu_i$. Consider the measures $\nu_i:=\psi_i|c_i'|\dd s\in \calM(S^1)$ with $\psi_i:=\sigma_i\circ T_{c_i}$. Assume 
    \begin{equation}\label{eq: wass}
        \inf_{\eta\in \Gamma_b(\sigma_0\mu_0, \sigma_1\mu_1)}\int_{S^1\times S^1} c(x,y)\dd \eta<+\infty.
    \end{equation}
    There exists a reparametrization $\phi\in \Pi(S^1)$ such that 
    \begin{equation}
        \eta:=(T_{c_0}\times T_{c_1\circ \phi})_\#\nu_0
    \end{equation}
    is a minimizer of \eqref{eq: wass}.
    Moreover it holds
    \begin{equation}\label{eq: phi-push}
        \psi_0(x)|c_0'(x)|= \psi_1(\phi(x)) |(c_1\circ \phi)'(x)| \qquad a.e. \quad x\in S^1.
    \end{equation}
    \begin{proof}
        By Lemma \ref{lemma: equivalence Gauss}, we can construct a minimizer of \eqref{eq: wass} from a minimizer of \eqref{eq: inf}.
        By Lemma \ref{lemma: Delon}, we have that 
        \begin{equation}
        \nu^*:=(q\times q)_\#\narrowcheck\nu_\theta, \qquad \narrowcheck{\nu}_{\theta}:=(F_{\widecheck{\nu}_0}^{-1}\times (F_{\widecheck{\nu}_1}-\theta)^{-1})_\#(\dd t\mres [0,1]),
    \end{equation}
        is a minimizer of \eqref{eq: inf}. Denote by 
        \begin{equation}\label{eq: phitheta}
            \narrowcheck{\phi}_\theta:=(F_{\widecheck{\nu}_1}+\theta)^{-1}\circ F_{\widecheck{\nu}_0}:\R\to \R
        \end{equation} 
        the monotone transport between the CDFs. Since $\narrowcheck{\nu}_i$ are absolutely continuous w.r.t Lebesgue and supported everywhere in $\R$, we get that $\narrowcheck{\phi}_\theta$ is a homeomorphism. By construction, the densities $|\narrowcheck{c}_i'|>0$ are bounded and $\psi_i\in L^1_+(S^1)$ are strictly positive almost everywhere. This gives that $F_{\widecheck{\nu}_i}\in W^{1,1}_{\loc}(\R)$ with $F'_{\widecheck{\nu}_i}>0$ a.e. $t\in \R$. The same holds for the inverse, and we obtain that $\narrowcheck{\phi}_\theta\in \Pi(\R)$. 
        Then, since $(F_{\widecheck{\nu}_0})_\#\narrowcheck{\nu}_0=\dd t$, we get that the measure $\narrowcheck{\nu}_\theta$ in Lemma \ref{lemma: Delon} can be equivalently described as
        \begin{equation}
            \narrowcheck{\nu}_\theta=(\Id\times \narrowcheck{\phi}_\theta)_\#\narrowcheck{\nu}_0. 
        \end{equation}
        Let $\phi\in \Pi(S^1)$ be the reparametrization induced by $\narrowcheck{\phi}_{\theta}$. Then, by the marginal constraints of $\nu^*$, we get that $\phi_\#\nu_0=\nu_1$, which corresponds to Equation \eqref{eq: phi-push}.
    \end{proof}
\end{proposition}

\begin{example}
    The assumption on the strict positivity of $\sigma_i$ is essential in the previous construction. As an example, 
    continuing on the setup of Example \ref{ex: seg-square}, recall that $\sigma_0\mu_0=\sigma_1\mu_1$. For convenience, assume that $c_0,c_1$ are constant speed parametrizations, which means that $\nu_0= \dd s$ and $\nu_1=\psi_1\dd s$. The density $\psi_1$ is non-zero and constant on the intervals $(0,\pi/2)$ and $(\pi,3\pi/2)$. Then, the map $\phi_\theta$ in Equation \eqref{eq: phitheta} maps, as expected, the first half of the circle to the first quarter and the second half to the third quarter. Therefore, it has jump discontinuity and an optimal reparametrization does not exist.
\end{example}

\subsection{Equivalence result}\label{sec: equiproof}
We are now in a position to combine the previous results. The lifting procedure, the characterization of optimal couplings of the lifted one-dimensional problem, and the monotonicity principle together show that optimal WFR couplings induce optimal SRVT reparametrizations, and conversely.
This leads to the following equivalence theorem.

\begin{theorem}\label{thm: elastic-WFR}
    Let $c_0, c_1\in \mathdiam{\calI}$, and let $\mu_i:=\mu_{c_i}$ be the corresponding length measures. Then, it holds
    \begin{equation}
        d_{\operatorname{SRVT}}(c_0,c_1)=\calU(\mu_0,\mu_1).
    \end{equation}
    \begin{proof}
        Explicitly, we want to show that
        \begin{equation}
            \inf_{\phi\in \Pi(S^1)}\int_{S^1}\left|\frac{c_0'}{\sqrt{|c_0'|}}- \frac{(c_1\circ \phi)'}{\sqrt{|(c_1\circ \phi)'|}}\right|^2 \dd s = \inf_{\gamma}\int_{\bC\times \bC} \left|x- y\right|^2\dd \gamma.
        \end{equation}
        To prove the first inequality ($\geq$), we construct a coupling $\gamma_\phi$ from a given reparametrization $\phi$ as the pushforward of $\dd s$ through the map $\Phi: S^1\to \bC\times\bC$ defined by 
        \begin{equation}
            \Phi:=\left(\frac{c_0'}{\sqrt{|c_0'|}}, \frac{(c_1\circ \phi)'}{\sqrt{|(c_1\circ \phi)'|}}\right).
        \end{equation}
        For any Borel set $U\subset S^1$, the second marginal constraint of  $\gamma_\phi:=\Phi_\#(\dd s)$ becomes
        \begin{equation}
            (\theta_1)_\#(\operatorname{r}_1^2\gamma_\phi)(U)=\int_{T_{c_1\circ\phi}^{-1}(U)} |(c_1\circ \phi)'|\;\dd t = \mu_1(U).
        \end{equation}
        Similarly the first marginal constraint gives $(\theta_0)_\#(\operatorname{r}_0^2\gamma_\phi)=\mu_0$.
        Therefore,
        \begin{equation}
            \int_{S^1}\left|\frac{c_0'}{\sqrt{|c_0'|}}- \frac{(c_1\circ \phi)'}{\sqrt{|(c_1\circ \phi)'|}}\right|^2 \dd s = \int_{\bC\times \bC} \left|x- y\right|^2\dd\gamma_\phi \geq \calU^2(\mu_0, \mu_1).
        \end{equation}

        To prove the second inequality ($\leq$), assume that $\mu_0$ and $\mu_1$ are weakly admissible in the sense of Definition \ref{def: weak-adms}. This assumption is satisfied when $c_0$ and $c_1$ are non-degenerate convex loops by Lemma \ref{lemma: non-deg, adm}, or one of the loops $c_i$ is degenerate in the direction $x$, but $\mu_j(\{\pm\;\i x\})=0$ by Corollary \ref{cor: deg-adms}.
        Consider the entropic formulation of Equation \eqref{eq: ent-non-deg}, defined in terms of the densities $\sigma_i\in L^1(S^1, \R_+)$ strictly positive $\mu_i$-almost everywhere, see Corollary \ref{cor: weak-adms}. Up to rescaling the measures $\sigma_i\mu_i$, we can use Proposition \ref{pro: final-eq} to get a reparametrization $\phi\in \Pi(S^1)$ such that
        \begin{equation}
            \eta^*:=(T_{c_0}\times T_{c_1\circ \phi})_\#(\psi_0|c_0'|\dd s).
        \end{equation}
        is an optimal coupling for
        \begin{equation}
            \inf_{\eta\in \Gamma_b(\sigma_0\mu_0, \sigma_1\mu_1)}\int_{S^1\times S^1}\ell(\operatorname{dist}(x,y))\dd\eta.
        \end{equation}
        Let $\sigma_0,\sigma_1$ be optimal for $\overline{\calF}$. By Proposition \ref{prop: ex log}, we get that 
        \begin{equation}
            \gamma^*:=\left(\frac{1}{\sqrt{\sigma_0}}\;T_{c_0}\times\frac{1}{\sqrt{\sigma_1}}\;T_{c_1\circ \phi}\right)_\#(\psi_0|c_0'|\;ds)
        \end{equation}
        is optimal for $\calU(\mu_0,\mu_1)$.
        Therefore,
        \begin{align}
            \int_{\bC\times \bC} |x-y|^2\dd\gamma^*&=\int_{S^1}\left|\frac{1}{\sqrt{\psi_0}}T_{c_0}-\frac{1}{\sqrt{\psi_1\circ \phi}}T_{c_1\circ \phi}\right|^2\psi_0|c_0'|\dd s\\
            &=\int_{S^1} \left|\sqrt{|c_0'|}T_{c_0}-\sqrt{\frac{\psi_0|c_0'|}{\psi_1\circ \phi}}T_{c_1\circ \phi}\right|^2\dd s,\\
            &=\int_{S^1} \left|\frac{c_0'}{\sqrt{|c_0'|}}-\frac{(c_1\circ \phi)'}{\sqrt{|(c_1\circ \phi)'|}}\right|^2\dd s
        \end{align}
        where we used the equivalence in Equation \eqref{eq: phi-push}. 
        
        Assume now that $c_0$ and $c_1$ are both degenerate convex loops that are orthogonal to each other. By point \ref{it: annih2} of Proposition \ref{prop: ex log}, the Wasserstein-Fisher-Rao distance between $\mu_0,\mu_1$ is
        \begin{equation}
            \calU^2(\mu_0,\mu_1)=|\mu_0|(S^1)+|\mu_1|(S^1).
        \end{equation}
        Moreover, by construction $c_0'\cdot c_1'=0$, which can be used to say
        \begin{equation}
            \int_{S^1}\left|\frac{c_0'}{\sqrt{|c_0'|}}- \frac{c_1'}{\sqrt{|c_1'|}}\right|^2 \dd s = \int_{S^1}|c_0'|\dd s+ \int_{S^1}|c_1'|\dd s= |\mu_0|(S^1)+|\mu_1|(S^1).
        \end{equation}
        In particular, this concludes the proof because $\Id\in \Pi(S^1)$ and thus
        \begin{equation}
            d^2_{\operatorname{SRVT}}(c_0,c_1)\leq |\mu_0|(S^1)+|\mu_1|(S^1)=\calU^2(\mu_0,\mu_1).
        \end{equation}
        
        For the last case, assume without loss of generality that $c_0$ is a degenerate loop, with $\mu_0=r_0\delta_{x_0}+r_0\delta_{-x_0}$ and assume that $c_1$ is non-degenerate with $\mu_1(\{\pm\;\i x_0\})>0$. 
        Let $\mu_0^k$ be a sequence of measures converging to $\mu_0$ defined in Lemma \ref{lemma: seq-weak-adms}, and $c_0^k\in \mathdiam{\calI}$ be a sequence of representatives from Theorem \ref{thm: one-to-one}.
        By construction, $\mu_0^k$, $\mu_1$ and $\mu_0^k, \mu_0$ are weakly admissible for all $k$, and we already proved the equivalence. So, by triangular inequality 
        \begin{equation}\label{eq: final-boss}
            d_{\operatorname{SRVT}}(c_0,c_1)\leq d_{\operatorname{SRVT}}(c_0,c_0^k)+d_{\operatorname{SRVT}}(c_0^k,c_1)\leq 2\calU(\mu_0^k,\mu_0)+\calU(\mu_0,\mu_1)
        \end{equation}
        Since $\mu_0^k\weak \mu_0$, and $|\mu_0^k|(S^1)=|\mu_0|(S^1)$, we also get $\calU(\mu_0^k,\mu_0)\to 0$, because $\calU$ metrizes the weak topology \cite[Theorem 7.15]{Liero_2017}. This implies the desired inequality in \eqref{eq: final-boss}.         
    \end{proof}
\end{theorem}
\begin{corollary}
    Let $c_0, c_1\in \mathdiam{\calI}$ be convex loops with weakly admissible length measures $\mu_0,\mu_1\in \calM^0(S^1)$. Then, there exists an optimal reparametrization $\phi\in \Pi(S^1)$ for $d_{\operatorname{SRVT}}(c_0,c_1)$.
\end{corollary}
The previous Corollary applies in particular to all pairs of non-degenerate convex loops.

\section{Towards optimization over length measures}\label{sec: opti}
{The results of Section \ref{sec: equi} have three consequences that motivate what follows. First by Theorem \ref{thm: elastic-WFR} the elastic distance between convex loops, which is defined through a nonconvex minimization over reparametrizations, equals $\calU(\mu_{c_0}, \mu_{c_1})$, and $\calU^2$ is jointly convex by Theorem \ref{thm: relaxation}. Since every $\mu\in \calM^0(S^1)\setminus\{0\}$ is a length measure of a unique convex loop (Theorem \ref{thm: one-to-one}), comparing convex shapes by the SRVT distance becomes a convex problem over the cone $\calM^0(S^1)$, improving the computational tractability. Second, atoms of $\mu_c$ correspond to edges of $c:$ finitely supported length measures are exactly those of convex polygons, and the number of atoms is the number of edges, which translates the notion of sparsity to polygons. Third, this formulation gives access to sparse optimization methods that rely on linear subproblems over WFR balls. These subproblems admit sparse solutions, that is, polygons with few edges (Theorem \ref{thm: J+2}).}

\subsection{Linear optimization over WFR balls}\label{sec: linearopt}
Following \cite{YueKuWie}, we prove a finiteness result for solutions of linear optimization problems constrained to Wasserstein-Fisher-Rao balls centered on a finite measure. We restrict the study to $S^1$ as base space for consistency with the rest of the paper, but all results in this section can be translated to any compact metric space. In fact, we believe that it would be possible to generalize also to any polish metric space, provided that the measures satisfy a bounded moment condition just like in the balanced case.
\begin{proposition}
    Let $\mu_1\in \calM_+(S^1)$ be any measure. Then the $\lambda$-WFR ball 
    \begin{equation}
        \calB_\lambda(\mu_1):=\{\mu_0\in \calM_+(S^1): \calU(\mu_0,\mu_1)\leq \lambda\}
    \end{equation}
    is weakly compact for any $\lambda\geq 0$. Moreover, the total mass of any $\mu_0\in \calB_{\lambda}(\mu_1)$ is bounded by 
    \begin{equation}\label{eq: bound-mass}
        |\mu_0|(S^1)\leq \left(\lambda+\sqrt{|\mu_1|(S^1)}\right)^2.
    \end{equation}
    \begin{proof}
        The ball $\calB_{\lambda}(\mu_1)$ is the $\lambda$-sublevelset of the map $\mu_0\mapsto \calU(\mu_0,\mu_1)$, which is weakly lower semicontinuous \cite{SavSod23}. In particular, $\calB_{\lambda}(\mu_1)$ is weakly closed. Since the base space is compact, it is also equally tight. By \cite[Theorem 2.2]{Liero_2017}, we obtain the first claim.
        By triangular inequality, we obtain the bound on the total mass as
        \begin{equation*}
            \sqrt{|\mu_0|(S^1)}=\calU(\mu_0,0)\leq \calU(\mu_0,\mu_1)+\calU(\mu_1,0)\leq \lambda+\sqrt{|\mu_1|(S^1)}.\qedhere
        \end{equation*}
    \end{proof}
\end{proposition}

Let $f:S^1\to \R$ be an upper semicontinuous function. Given $\mu_1\in \calM_+(S^1)$, we want to study the properties of 
\begin{equation}\label{eq:linear-problem}
    \sup_{\mu_0\in \calB_{\lambda}(\mu_1)} \int_{S^1} f\dd \mu_0 \tag{LP}.
\end{equation}

\begin{proposition}\label{prop: ex-sol}
    Problem \eqref{eq:linear-problem} is finite and it admits a solution.
    \begin{proof}
        Since $\calU$ metrizes the weak topology \cite[Theorem 7.15]{Liero_2017}, and $\calD_+(S^1)$ is weakly dense in $\calM_+(S^1)$ \cite[Proposition 2.1]{SavSod23}, there exists $\nu_1:=\sum_{j=1}^J r_1^j\delta_{x_1^j}\in \calD_+(S^1)$ with $|\nu_1|(S^1)=|\mu_1|(S^1)$ sufficiently close to $\mu_1$, i.e. $\calU(\nu_1,\mu_1)<\lambda/2$. In particular, 
        for any $\mu_0\in \calB_{\lambda}(\mu_1$), it holds
        \begin{equation}
            \int_{S^1}f \dd\mu_0\leq \|f\|_{L^\infty} |\mu_0|(S^1)\leq \|f\|_{L^\infty}\left(\lambda+\sqrt{|\mu_1|(S^1)}\right)^2<+\infty.
        \end{equation}
        We want to show that the map $\mu_0\mapsto \int_{S^1}f\dd\mu_0$ is upper semicontinuous. Let $\{\mu_0^k\}_k\subset \calB_{\lambda}(\mu_1)$ {be} a sequence converging weakly to some $\mu_0^\infty\in \calB_{\lambda}(\mu_1)$. We approximate $f$ via the continuous functions $f_n(x):=\sup_{y\in S^1} [f(y)-n\operatorname{dist}(x,y)$], where $f_n\to f$ pointwise for $n\to \infty$, and $f_n\geq f$. By weak convergence of $\mu_0^k$, we get 
        \begin{equation}
            \limsup_{k\to \infty}\int_{S^1}f\dd\mu_0^k\leq \lim_{k\to \infty}\int_{S^1}f_n\dd\mu_0^k=\int_{S^1}f_n\dd\mu_0^\infty.
        \end{equation}
        By monotone convergence, the right-hand side converges to $\int_{S^1} f\dd\mu_0^\infty$, which gives the desired upper semicontinuity property. By Weierstrass Theorem we obtain the claim.
    \end{proof}
\end{proposition}

\begin{theorem}\label{thm: J+2}
    Let $\mu_1\in \calD^N_+(S^1)$ be a finitely supported measure on $S^1$ with exactly $N$ atoms. Then, problem \eqref{eq:linear-problem} admits a solution $\mu_0\in \calD_+(S^1)$ with at most ${N+1}$ support points.
    \begin{proof}
        We can rewrite problem \eqref{eq:linear-problem}, in terms of homogeneous couplings as
        \begin{equation}\label{eq: sup-coup-or}
            \sup_{\gamma\in \calP_2(B_R\times B_R)} \int_{\bC\times \bC} |x|^2 f\left(\theta(x)\right)\dd\gamma(x,y)
        \end{equation}
        subject to $(\theta_1)_\#(\operatorname{r}_1^2\gamma)=\mu_1$ and the bound on the Fisher-Rao energy $J(\gamma)\leq \lambda^2$ with $R=\sqrt{2|\mu_1|(S^1)+\lambda^2+2\lambda\sqrt{|\mu_1|(S^1)}}$. Note that the restriction to $B_R\times B_R$ makes sense in view of Proposition \ref{prop: basicWFR} and the bound on the mass in Equation \eqref{eq: bound-mass}.
        {We can drop the mass constraint on the coupling by a standard dilation argument. In particular, \eqref{eq: sup-coup-or} has the equivalent formulation
        \begin{equation}\label{eq: sup-coup}
            \sup_{\gamma\in \calM_2(\Omega)} \int_{\bC\times \bC} |x|^2 f\left(\theta(x)\right)\dd\gamma(x,y)\tag{CP}
        \end{equation}
        subject to $(\theta_1)_\#(\operatorname{r}_1^2\gamma)=\mu_1$ and  $J(\gamma)\leq \lambda^2$, with 
        \begin{equation}
            \Omega:=\{(x,y)\in \bC\times \bC: h^2(x,y):=|y|^2+|x-y|^2=1\}.
        \end{equation}
        The energy constraint already gives a bound on the mass
        \begin{equation}\label{eq: bound2}
            \int_{\bC\times \bC}\dd \gamma=|\mu_1|(S^1)+J(\gamma)\leq |\mu_1|(S^1)+\lambda^2,
        \end{equation}
        and, to pass from competitors of \eqref{eq: sup-coup-or} to \eqref{eq: sup-coup}, one defines 
        \begin{equation}
            \tilde\gamma=\operatorname{dil}_h \gamma:=(\operatorname{prd}_h)_\#(h^2\gamma), \qquad \operatorname{prd}_h(x,y):=\left(\frac{x}{h}, \frac{y}{h}\right)
        \end{equation}
        which is supported on $\Omega$ and maintains the marginals. Conversely, one defines $\operatorname{dil}_{\rho/\sqrt{M}}\gamma$ with $\rho^2(x,y):=|x|^2+|y|^2$ and $M=\int\rho^2\dd\gamma$.
        }
        Since $\mu_1$ is of the type $\sum_{j=1}^Jr_1^j\delta_{x_1^j}$, we can rewrite the marginal constraints as 
        \begin{equation}
            r_1^j=\int_{\bC\times \bC} |y|^2 \one _{\bC\times \theta^{-1}(x_1^j)}(x,y)\dd\gamma(x,y)
        \end{equation}
        Let us denote by $g_j:\bC\times \bC\to \R$ the function $g_j(x,y):=|y|^2 \one _{\bC\times \theta^{-1}(x_1^j)}(x,y)$. The full set of constraints on $\gamma$ are ${N+1}$ equations that read
        \begin{equation}\label{eq: cm1}
            \int_{\bC\times \bC}|x-y|^2 \dd\gamma\leq\lambda^2, \qquad \int_{\bC\times \bC} g_j(x,y)\dd\gamma=r_1^j \quad \text{ for } j=1,...,N.\tag{C1}
        \end{equation}
        We want to prove that problem \eqref{eq: sup-coup} is equivalent to minimizing over discrete couplings in ${\calD^{N+1}_{+}(\Omega)}$ with ${N+1}$ atoms. 
        To do so, we first look at problem \eqref{eq: sup-coup}, with the following restricted set of constraints, for $0\leq \alpha\leq \lambda$,
        \begin{equation}\label{eq: cm2}
            \int_{\bC\times \bC}|x-y|^2 \dd\gamma=\alpha^2, \qquad \int_{\bC\times \bC} g_j(x,y)\dd\gamma=r_1^j \quad \text{ for } j=1,...,N.\tag{C2}
        \end{equation}
        The only difference from the constraints in \eqref{eq: cm1} is the equivalence $J(\gamma)=\alpha^2$. The first marginal of a competitor of \eqref{eq: sup-coup} with constraints \eqref{eq: cm2} lives in the $\alpha$ ball $\calB_\alpha(\mu_1)$. 
        The reason why we restricted the problem to \eqref{eq: cm2} is to use the standard theory of infinite dimensional linear programming. \cite[Appendix B, Proposition 1]{YueKuWie} ensures that the continuous problem \eqref{eq: sup-coup} with constraints $\eqref{eq: cm2}$ is equivalent to the discrete problem
        \begin{equation}\label{eq: relax}
            \sup_{\gamma\in {\calD_{+}^{N+1}(\Omega)}} \int_{\bC\times \bC} |x|^2 f\left(\theta(x)\right)\dd\gamma(x,y)\tag{DP}
        \end{equation}
        subject to \eqref{eq: cm2}.
        Note that \eqref{eq: cm2} is a subset of \eqref{eq: cm1}, so competitors of \eqref{eq: sup-coup} (or \eqref{eq: relax}) subject \eqref{eq: cm2} are also competitors of \eqref{eq: sup-coup} (or \eqref{eq: relax}) subject to \eqref{eq: cm1}, giving the inequalities $\eqref{eq: sup-coup}\eqref{eq: cm2}\leq \eqref{eq: sup-coup}\eqref{eq: cm1}$ and $\eqref{eq: relax}\eqref{eq: cm2}\leq \eqref{eq: relax}\eqref{eq: cm1}$. All together, it holds 
        \begin{equation}
            \eqref{eq: sup-coup}\eqref{eq: cm2}=\eqref{eq: relax}\eqref{eq: cm2}\leq \eqref{eq: relax}\eqref{eq: cm1} \leq\eqref{eq: sup-coup}\eqref{eq: cm1}.
        \end{equation}
        
        Before commenting on the equivalence between \eqref{eq: sup-coup}\eqref{eq: cm1} and \eqref{eq: sup-coup}\eqref{eq: cm2}, we want to prove the existence of solutions of \eqref{eq: relax}\eqref{eq: cm1}.
        The feasible region for $\gamma$ is the intersection of these three sets: 
        \begin{align}\label{eq: feasible}
            &\calF_1:={\calD_{+}^{N+1}(\Omega)};\\
            &\calF_2:=\{\gamma\in {\calM_+(\Omega)}: (\theta_1)_\#(\operatorname{r}_1^2\gamma)=\mu_1\};\\
            &\calF_3:=\left\{\gamma\in \calM_+(\bC\times \bC): \int_{\bC\times \bC}|x-y|^2 \dd\gamma\leq \lambda^2\right\}.
        \end{align}
         Note that $\calF_1$ is narrowly closed because {$\Omega$ is compact}. $\calF_2$ is the preimage of $\mu_1$ under $(\theta_1)_\#\operatorname{r}_1^2$, which is narrowly continuous {because $\operatorname{r}_1$ is continuous and bounded in $\Omega$}. $\calF_3$ is the {sublevel set} of the lower semicontin{u}ous energy $J$. 
         The intersection $\calF_1\cap \calF_2\cap \calF_3$ is tight because of {compactness of $\Omega$} and the bound in total mass in Equation \eqref{eq: bound2}. Thus, the feasible region is narrowly compact. 
         Moreover, the map
         \begin{equation}
             \gamma\mapsto\int_{\bC\times \bC} |x|^2 f\left(\theta(x)\right)\dd\gamma(x,y)
         \end{equation}
         is narrowly upper semicontinuous in the feasible region, which can be proven with the same arguments in the proof of Proposition \ref{prop: ex-sol}. Thus, the discrete problem \eqref{eq: relax} subject to \eqref{eq: cm1} admits a solution.

         The last thing to prove is whether \eqref{eq: sup-coup}\eqref{eq: cm1} and \eqref{eq: sup-coup}\eqref{eq: cm2} coincide. Let $\gamma$ be a solution of \eqref{eq: sup-coup} subject to \eqref{eq: cm1}. We know it exists because of Proposition \ref{prop: ex-sol} and \ref{prop: basicWFR}. Define $0\leq \alpha\leq \lambda$ such that $J(\gamma)=\alpha^2$. Then $\gamma$ is also a solution to \eqref{eq: sup-coup}\eqref{eq: cm2}. 

        We just proved the existence of a discrete measure $\gamma\in {\calD^{N+1}_{+}(\Omega)}$ that is solution of \eqref{eq: sup-coup} subject to \eqref{eq: cm1}. Taking its first marginal $\mu_0=(\theta_0)_\#(\operatorname{r}_0^2\gamma)$ gives the desired solution to \eqref{eq:linear-problem}.
    \end{proof}
   
\end{theorem}

\begin{corollary}\label{cor: J+3}
    Let $\mu_1\in \calD_+(S^1)$ be a finitely supported measure on $S^1$ with exactly $N$ atoms with vanishing first moment. Then, problem \eqref{eq:linear-problem} admits a solution $\mu_0\in \calD_+(S^1)\cap \calM^0(S^1)$ with at most $N+{3}$ support points.
    \begin{proof}
Two additional scalar constraints are added to \eqref{eq: sup-coup}, namely the real and imaginary
parts of the vanishing first moment condition
\[
\int_{S^1} x\, d\mu_0(x)=0\in\mathbb C.
\]
Using $\mu_0 = (\theta)_\#(r_0^2\gamma)$ and $r_0(x)=|x|$, we compute
\[
\int_{S^1} x\, d\mu_0(x)
=\int_{\mathbb C\times\mathbb C} |x|^2\theta(x)\, d\gamma(x,y)
=\int_{\mathbb C\times\mathbb C} |x|x\, d\gamma(x,y),
\]
where $|x|^2\theta(x)=|x|x$ for $x\neq 0$ (and equals $0$ at $x=0$).
Let $R=\sqrt{2|\mu_1|(S^1)+\lambda^2+2\lambda\sqrt{|\mu_1|(S^1)}}$ be as in the proof of Theorem \ref{thm: J+2}, so that the feasible
couplings are supported in $ B_R\times B_R$. Define $q: B_R\to\mathbb C$ by
$q(x)=|x|x$, and let $q_0,q_1$ be its real and imaginary parts. Then $q_0,q_1$ are bounded and
continuous on $ B_R$, hence $(x,y)\mapsto q_i(x)$ are bounded and continuous on
$ B_R\times B_R$.
Therefore, for any narrowly convergent sequence $\gamma_k\rightharpoonup \gamma$ in
$\mathcal P( B_R\times B_R)$,
\[
\lim_{k\to\infty}\int_{ B_R\times B_R} q_i(x)\, d\gamma_k(x,y)
=
\int_{ B_R\times B_R} q_i(x)\, d\gamma(x,y), \qquad i=0,1.
\]
Hence the set of couplings satisfying the additional moment constraints is closed under narrow
convergence. Intersecting with the compact feasible region from Theorem \ref{thm: J+2} yields a compact feasible
set, so the remainder of the LP argument applies and gives the claimed sparsity bound.
\end{proof}
\end{corollary}

\subsection{{Suitability to conditional gradient methods}}\label{sec: homogenized}
In view of unifying the Wasserstein-Fisher-Rao distance and generalized conditional gradient methods, we introduce the following ``homogenized'' energy:
\begin{definition}\label{def: homogenized}
    Given a measure $\mu_1\in \M_+(S^1)$, we define the energy $\bS_{\mu_1}: \M(S^1)\to [0,+\infty]$ as
    \begin{equation}
        \bS_{\mu_1}(\mu_0):=\frac{|\mu_0|(S^1)}{2}\calU^2\left(\frac{\mu_0}{|\mu_0|(S^1)}, \frac{\mu_1}{|\mu_1|(S^1)}\right)
    \end{equation}
\end{definition}
with $\bS_{\mu_1}(0)=0$. Equivalently, we can write $\bS_{\mu_1}$ as 
\begin{equation}
    \bS_{\mu_1}(\mu_0)= \frac{1}{2}\calU^2 \left(\mu_0, \frac{|\mu_0|(S^1)}{|\mu_1|(S^1)}\mu_1\right).
\end{equation}
In any case, there is a rescaling so that the mass between the source and the target is equal. Note that the operator is independent of the mass of the target measure $\mu_1$. For simplicity, one could assume $\mu_1$ to be a probability measure.

\begin{proposition}
    For $\mu_1\in \M_+(S^1)$, $\bS_{\mu_1}$ is a convex and positively one-homogeneous functional. 
    \begin{proof}
        The one-homogeneous property holds by construction, so let us focus on the convexity. Outside of $\M_+(S^1)$, the energy $\bS_{\mu_1}$ is infinite, so we restrict our study to $\M_+(S^1)$, which is a convex cone in $\M(S^1)$. Consider the map $\Phi: \M_+(S^1)\to \R$ defined by 
        \begin{equation}
            \Phi(\mu_0):=\frac{1}{2}\calU^2 \left(\mu_0, \frac{\mu_1}{|\mu_1|(S^1)}\right).
        \end{equation}
        It is a convex map, thanks to Theorem \ref{thm: relaxation}. Now consider the function $\Psi: \M_+(S^1)\times \R_+\to \R$ given by
        \begin{equation}
            \Psi(\mu_0, t):=t \Phi\left(\frac{\mu_0}{t}\right).
        \end{equation}
        This is also a convex map thanks to \cite[Lemma 2.1]{IwaOnni09}. Moreover, we see that the energy $\bS_{\mu_1}$ is the composition of the map $\Psi$ with the inclusion $\iota: \M_+(S^1)\hookrightarrow \M_+(S^1)\times \R_+$ where 
        \begin{equation}
            \iota(\mu_0):=(\mu_0, |\mu_0|(S^1)).
        \end{equation}
        On positive measures, the inclusion $\iota$ is linear, making $\bS_{\mu_1}=\Psi\circ \iota$ a convex mapping.
    \end{proof}
\end{proposition}
We mentioned in Theorem \ref{thm: relaxation} (proved in \cite{SavSod23}) that the energy $\calU^2$ is the closed convex envelope of its restriction to the space $\Delta_+(S^1)\times \Delta_+(S^1)$ of pairs of point masses.
{This suggests a particularly simple atomic structure for the extreme points of its unit sublevel set in $\calM_+(S^1)\times\calM_+(S^1)$.} {It is then natural to ask whether a similar characterization holds for} the operator $\mathbb S_{\mu_1}${, with extreme points supported on $\operatorname{supp}(\mu_1)$ plus at most one additional point}.
{
\begin{lemma}\label{lemma: counter}
    Let $\nu= \delta_{x_0}$ and $\calB_\nu:=\{\mu\in \calM_+(S^1): \mathbb S_\nu(\mu)\leq 1 \}$. Then $\cal B_\nu$ admits extreme points supported on two points outside $\supp(\nu)$.
    \begin{proof}
        Choose $x_1,x_2\neq x_0$ with $0<d_i:=\operatorname{dist}(x_0,x_i)<\tfrac\pi 2$ and distinct $p_i:=\cos^2 d_i$. Write $a=(a_1,a_2)\in \R^2_+$, and define $A(a)=a_1+a_2$ and $B(a)=a_1p_1+a_2p_2$. Since the reference is a single Dirac, we get the explicit formula from Equation \eqref{eq:single-dirac-hom}
        \begin{equation}
            \mathbb S_\nu(a_1\delta_{x_1}+a_2\delta_{x_2})=A(a)-\sqrt{A(a)B(a)}.
        \end{equation}
        The map $g:a\mapsto \sqrt{A(a)B(a)}$ is concave, meaning that 
        \begin{equation}
            g(ta+(1-t)b)\geq t g(a)+(1-t)g(b),
        \end{equation}
        with equality if and only if $B(a)/A(a)=B(b)/A(b)$. Since $p_1\neq p_2$, we get that $S:a\mapsto A(a)-g(a)$ is strictly convex on the line-segment between two non-proportional points. In particular, for every $a,b$ such that $S(a)=S(b)=1$, we get that the segment joining them lies entirely inside the unit ball, meaning that every point at the boundary is an extreme point. This translates to $\bS_\nu$ and in particular to the intersection of $B_\nu$ with the set of such double Dirac deltas. Moreover, these extreme points remain extreme in $B_\nu$. indeed, if $\mu=a_1\delta_{x_1}+a_2\delta_{x_2}=t\mu'+(1-t)\mu''$ with $\mu',\mu''\in B_\nu$ and $t\in(0,1)$, nonnegativity forces $\mu'$ and $\mu''$ to be supported only on ${x_1,x_2}$. Their coefficient vectors therefore belong to the two-dimensional sublevel set, where extremality implies $\mu'=\mu''=\mu$.
    \end{proof}
\end{lemma}
The previous lemma shows that, even for a single Dirac reference measure, the unit ball admits extreme points whose supports are not contained in the support of the reference measure plus one additional point. This rules out the anticipated simplification of the linear minimization oracle of the type \eqref{eq:linear-problem} in generalized conditional gradient methods and complicates a complete characterization of the extreme points. 
Nevertheless, the convexity and positive one-homogeneity of $\bS_\nu$, together with the sparsity results of Theorem \ref{thm: J+2} and Corollary \ref{cor: J+3}, provide a promising starting point for such methods. Since $\bS_\nu$ vanishes precisely on the ray ${t\nu\geq 0}$, suitable treatment of this null direction is required for the well-poseness of the linear subproblem. We leave the practical development of corresponding sparse algorithms for future work.
}

\section{{Numerical experiments}}\label{sec: numer}
{
Theorem \ref{thm: elastic-WFR} offers two ways of computing $d_{\operatorname{SRVT}}(c_0,c_1)$ for convex loops.
The classical one optimizes over reparametrizations: numerically, $\phi$ is discretized as a monotone path on an
$N\times N$ grid and found by dynamic programming using the Python library fdasrsf \cite{fdasrsf}, the starting point of the closed curve
being searched exhaustively. The second one computes $\calU(\mu_0,\mu_1)$. If $c_0,c_1$ are polygons with $n$ and $m$ edges, then
$\mu_0=\sum_i a_i\delta_{x_i}$ and $\mu_1=\sum_j b_j\delta_{y_j}$ are atomic and, by Proposition \ref{prop: ex log},
\begin{equation}
  \calU^2(\mu_0,\mu_1)=\min_{\eta\in\R_+^{n\times m}} \sum_{i,j}\ell(\operatorname{dist}(x_i,y_j))\,\eta_{ij}
  +\sum_{i}\Big(r_i\log\tfrac{r_i}{a_i}-r_i+a_i\Big)+\sum_{j}\Big(s_j\log\tfrac{s_j}{b_j}-s_j+b_j\Big),
\end{equation}
with $r=\eta\one$ and $s=\eta^{\top}\one$; only pairs with $\operatorname{dist}(x_i,y_j)<\pi/2$ enter. We solve this
convex problem with L-BFGS-B \cite{chapel2021unbalanced} and, for comparison, with entropic unbalanced Sinkhorn \cite{chizat2019unbalancedoptimaltransportdynamic} using the Python library POT \cite{flamary2021pot} in both cases. Following the proof of
Proposition \ref{pro: final-eq}, we then recover $\phi^*=(F_{\narrowcheck\nu_1}-\theta)^{-1}\circ F_{\narrowcheck\nu_0}$ from the optimal densities $\sigma_0=r/a$, $\sigma_1=s/b$.
}

\begin{figure}[!t]
    \centering
    \includegraphics[width=\linewidth]{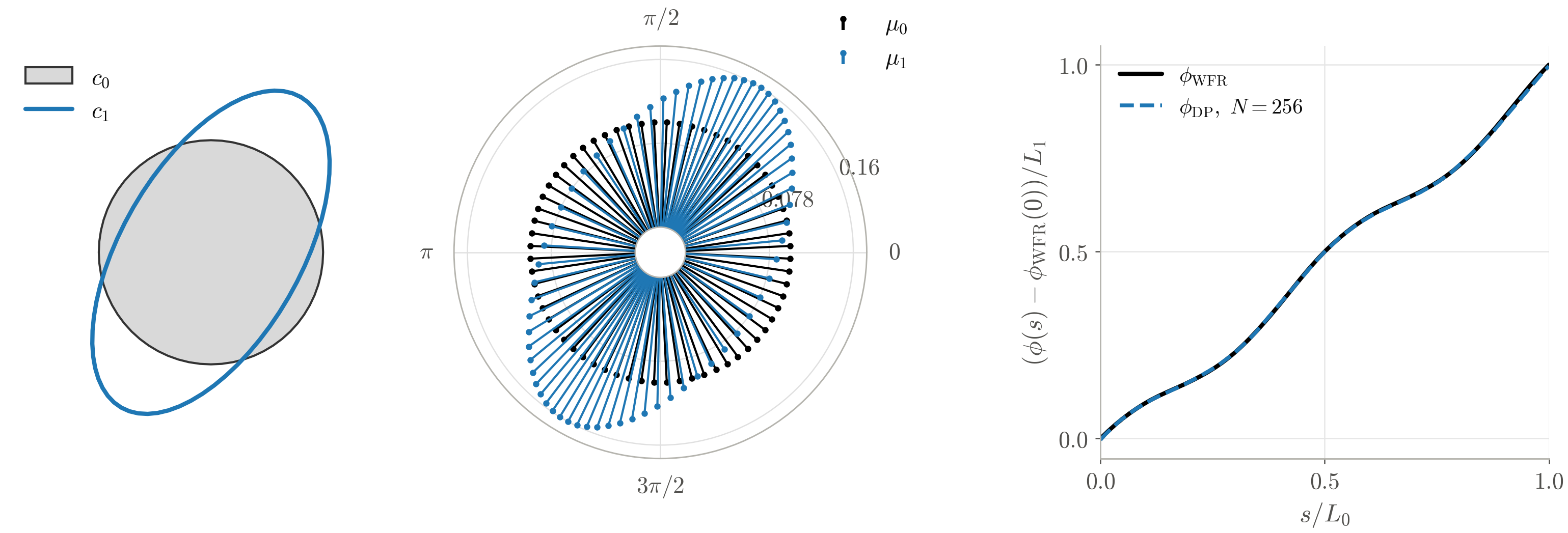}
    \includegraphics[width=\linewidth]{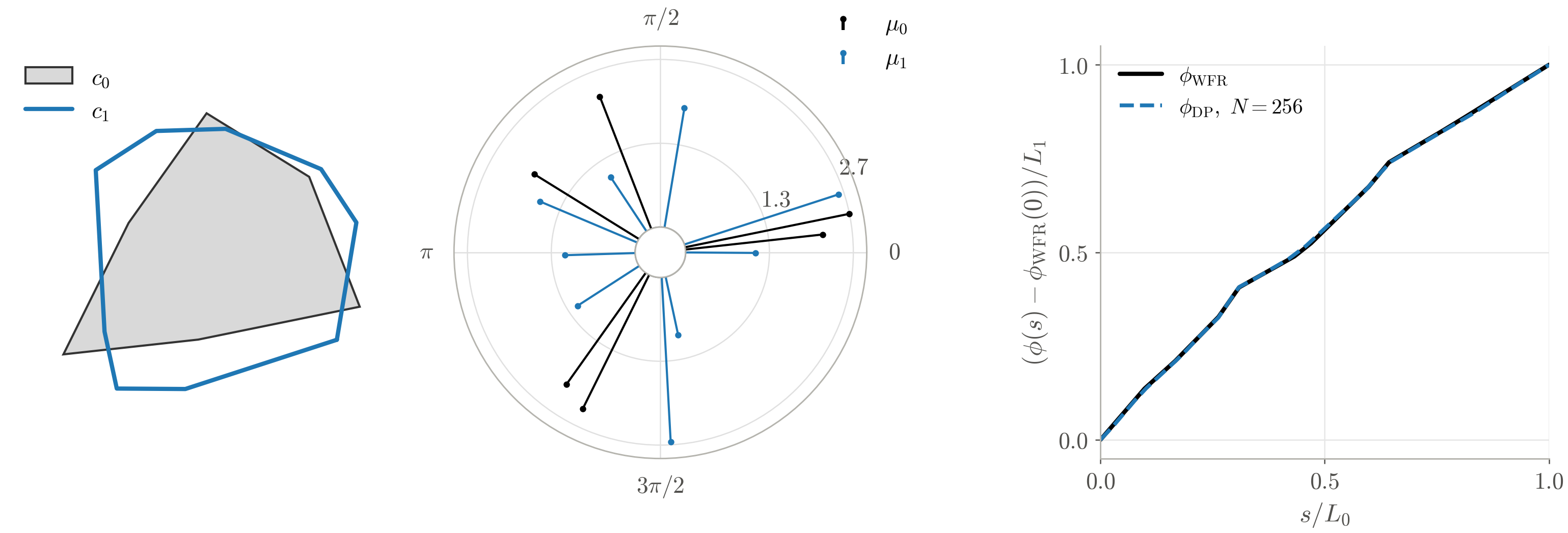}
    \includegraphics[width=\linewidth]{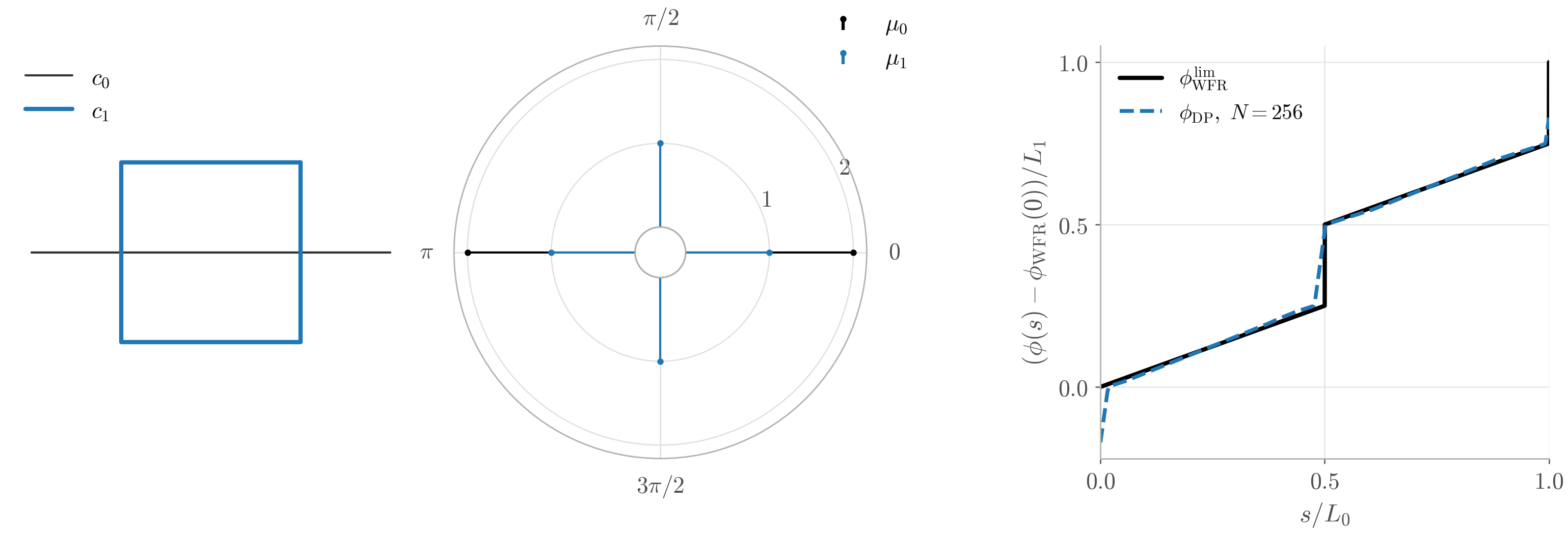}
    \caption{{From the top: an ellipse vs a circle, whose length measures are both discretized in 64 uniformly distributed atoms. Two random polygons with 6 and 9 edges, respectively. Lastly, the degenerate setting of Example \ref{ex: seg-square}. On every row, the leftmost picture depicts the shapes, while the middle one their respective length measures, and the rightmost graph shows the optimal reparametrization computed through WFR and SRV (DP), where $L_i$ are the total lengths.}}
    \label{fig:num-ex}
\end{figure}

\begin{table}[ht]
\centering
\small
\begin{tabular}{@{}llrrrrrr@{}}
\hline
Example & $(n,m)$ &\multicolumn{2}{c}{$\calU^2$, Unregularized} & \multicolumn{2}{c}{$\calU^2$, Sinkhorn} & \multicolumn{2}{c}{DP ($N=256$)} \\
&& Value & $t$ [s] & Value & $t$ [s] & Value & $t$ [s] \\
\hline
Ellipses & $(64,64)$&0.52115542 & 0.2234 & 0.54502171 & 0.02951 & 0.52783490 & 14 \\
Ellipses & $(256,256)$&0.51202419 & 19.16 & 0.54522522 & 0.05887 & 0.51799270 & 13.98 \\
Polygons& $(6,9)$ &2.47399278 & 0.007795 & 2.47556835 & 0.02758 & 2.55427822 & 14.44 \\
Degenerate &$(2,4)$& 2.34314575 & 0.0009012 & 2.34314622 & 0.02863 & 2.51668942 & 14.65 \\
\hline
\end{tabular}
\caption{{Squared-distance objective values and solver runtimes in seconds. Rows (first and second together) correspond to the examples provided in Figure \ref{fig:num-ex}, respectively. Sinkhorn: $\varepsilon=0.01$; DP: $N=256$. Exact runtimes exclude reparametrization recovery. Run on CPU M1 Max 10-core 2021.}}
\label{tab: numerics}
\end{table}

{The two approaches have very different computational costs. Dynamic programming works on an $N\times N$ grid with a fixed set of admissible slopes, and repeating it for every starting point of the closed curve requires $O(N^3)$ operations.
Its value is the energy of an admissible reparametrization, hence an upper bound for $d^2_{\operatorname{SRVT}}$. The problem for $\calU^2$ instead has $nm$ unknowns, independent of any discretization of the curves, and each iteration of L-BFGS-B and of Sinkhorn costs $O(nm)$, although they differ in overall complexity. In Table \ref{tab: numerics}, the unregularized solution takes between $1$\,ms and $19$\,s, that is up to $16000$ times less than dynamic programming with $N=256$, which needs about $14$\,s in each example. When the number of atoms is comparable to the DP grid size, L-BFGS-B takes $19$s, showing worse total runtime than DP. 
Dynamic programming also stays above the exact value: by $1.3\%$ for the ellipse and the circle, by $3.2\%$ for the random polygons, and by $7.4\%$ in the degenerate case, where the infimum is not attained and approaching it requires arbitrarily large slopes. Sinkhorn takes about $60$ms in the $(256,256)$ example and $30$\,ms in all three other cases (presumably overhead CPU time), which is significantly faster than the other solvers, but the entropic regularization biases the value.
}

\appendix
\section{Extension of the local optimality theory to extended-valued costs}\label{sec: app1}

In this appendix we extend the structural results of
\cite[Section~4 and Theorem~5.4]{Delon_2010}
to the case where the cost on $\mathbb R^2$ is allowed to take the value $+\infty$ outside a finite strip.
The goal is to justify the use of \cite[Lemma~4.13]{Delon_2010}
in the proof of Lemma~\ref{lemma: Delon}
for the lifted cost $\narrowcheck c_T$.

Let $c:\mathbb R^2\to[0,+\infty]$ be a Borel function satisfying:

\begin{enumerate}
\item[(H1)] (\emph{Finite strip}) There exists $\Lambda>0$ such that
\[
c(x,y)<\infty \quad\Longleftrightarrow\quad |x-y|<\Lambda.
\]
\item[(H2)] (\emph{Regularity in the finite region})
$c$ is lower semicontinuous and satisfies the Monge inequality
\[
c(x_1,y_1)+c(x_2,y_2)
\le
c(x_1,y_2)+c(x_2,y_1)
\]
whenever $x_1\leq x_2$, $y_1\leq y_2$ and all four quantities are finite.
\end{enumerate}

We consider periodic, locally finite transport plans $\gamma$
on $\mathbb R^2$ in the sense of \cite{Delon_2010},
i.e.
\[
\gamma(A+(2\pi k,2\pi k))=\gamma(A)
\]
for all Borel $A\subset\mathbb R^2$,
with marginals periodic lifts of probability measures on $S^1$.

\begin{definition}
Let $\gamma$ be such a plan.
We call $\gamma'$ a local modification of $\gamma$
if:
\begin{enumerate}
\item $\gamma'-\gamma$ is a compactly supported finite signed measure,
\item $\gamma'$ has the same marginals as $\gamma$,
\item $c\in L^1(\R^2, |\gamma'-\gamma|)$.
\end{enumerate}
A local modification is cost-reducing, or cost-equivalent if, respectively
\begin{equation}
    \int_{\R^2}c\dd (\gamma'-\gamma)<0, \qquad \int_{\R^2}c\dd (\gamma'-\gamma)=0.
\end{equation}
\end{definition}

\begin{definition}
A plan $\gamma$ is $c$-locally optimal
if it admits no cost-reducing  local modification.
\end{definition}

Let $\hat c(x,y):=\inf_{k\in\mathbb Z} c(x,y+2\pi k)$
be the induced cost on $S^1\times S^1$.
Let $\nu^*$ minimize
\[
\inf_{\nu\in\Gamma_b(\nu_0,\nu_1)} \int \hat c\, d\nu.
\]
where $\nu_0,\nu_1\in \calP(S^1)$ are absolutely continuous probability measures, and assume that this value is finite, which in turn implies that $\hat c$ is finite $\nu^*$-almost everywhere. Denoting by $\hat \Omega \subset S^1 \times S^1$ the set where $\hat c$ is finite, we can consider the pushforward $\iota_{\#}(\nu^* \mres \hat \Omega)$ th{r}ough the identification $\iota: S^1 \times S^1 \to [0,2\pi)^2$ and extend it periodically to produce a lift $\gamma^* \in \mathcal{M}_+(\R^2 \times \R^2)$ with
\[
\mathrm{supp}(\gamma^*)\subset\operatorname{cl} \Omega, \qquad\qquad  \Omega:=\{ |x-y|<\Lambda\}.
\]
Moreover
\[
\int_{[0,2\pi)^2} c\, d\gamma^*
=
\int_{S^1\times S^1} \hat c\, d\nu^*.
\]

By \cite[Lemma~4.4]{Delon_2010}, a result that is not dependent on the properties of the cost,
there exists a plan $\zeta^*\in\Gamma_b(dt,dt)$ such that
\[
\gamma^*
=
(F_{\gamma_0}^{[-1]}\times F_{\gamma_1}^{[-1]})_\# \zeta^*,
\]
where we denoted by $\gamma_i$ the marginals of $\gamma^*$. Since $\gamma^*$ is supported in $\operatorname{cl}\Omega$,
the conjugated cost
\[
\tilde c(t,s)
:=
c(F_{\gamma_0}^{[-1]}(t),F_{\gamma_1}^{[-1]}(s))
\]
is finite $\zeta^*$ almost everywhere.

\begin{lemma}
Under the above assumptions,
$\zeta^*$ is $\tilde c$-locally optimal.
\end{lemma}

\begin{proof}
Let $\zeta'$ be a local modification of $\zeta^*$ and define
\begin{equation}
\gamma' := (F_{\gamma_0}^{[-1]}\times F_{\gamma_1}^{[-1]})_\#\zeta'.
\end{equation}
By construction, $\gamma'$ has the same marginals as $\gamma^*$ and $\gamma'-\gamma^*$ is a compactly supported finite signed measure.
Let $\nu'$ be the projection of $\gamma'$ on the circle. Since $\nu^*$ minimizes $\int_{S^1\times S^1} \hat c\, d\nu$ and $\gamma^*$ is supported on the set where $c$ realizes $\hat c$, we obtain
\begin{equation}
\int_{\R^2} c \, d(\gamma'-\gamma^*) \ge 0.
\end{equation}

Using the definition
\begin{equation}
\tilde c(t,s) := c\big(F_{\gamma_0}^{[-1]}(t),F_{\gamma_1}^{[-1]}(s)\big),
\end{equation}
and the relation between $\gamma'$ and $\zeta'$, we rewrite the previous inequality as
\begin{equation}
\int_{\R^2} \tilde c \, d(\zeta'-\zeta^*) \ge 0.
\end{equation}
Hence $\zeta^*$ admits no cost-reducing local modification for $\tilde c$, and therefore is $\tilde c$-locally optimal.
\end{proof}

Define the maps
        \begin{equation}
            r_{\zeta^*}(x,y):=\zeta^*((-\infty, x]\times (y,\infty)), \qquad l_{\zeta^*}(x,y):=\zeta^*((x,\infty)\times (-\infty, y]),
        \end{equation}
        which measure the amount of mass transported from the left (right) of $x$ to above (below) $y$, respectively. By \cite[Lemma 4.8]{Delon_2010}, there exist maps $w_{\zeta^*}, m_{\zeta^*}:\R\to \R$ with $w_{\zeta^*}$ increasing and such that 
        \begin{equation}
            r_{\zeta^*}(x, w_{\zeta^*}(x))=l_{\zeta^*}(x, w_{\zeta^*}(x))=m_{\zeta^*}(x).
        \end{equation}
A desired property of monotone couplings is $m_{\zeta^*}(x)=0$ for all $x$, so that zero mass is crossing and moving from $(-\infty, x]$ to $(w_{\zeta^*}(x),+\infty)$ and from $(x, +\infty)$ to $(-\infty, w_{\zeta^*}(x)]$. This motivates the next result, which is a strengthened version of \cite[Lemma 4.9]{Delon_2010} reflecting that in our case the modifications must be concentrated in the strip where $c$ is finite.
\begin{lemma}\label{lemma: app2}
    For every $x\in \R$ there exists a cost-equivalent local modification $\zeta_x\in \Gamma_b(\dd t, \dd t)$ of $\zeta^*$ with $w_{\zeta^*}(x)=w_{\zeta_x}(x)$, and $m_{\zeta_x}(x)=0$.
    \begin{proof}
        Assume $m:= m_{\zeta^*}(x)>0$ and let $y=w_{\zeta^*}(x)$. By assumption, 
        \begin{equation}
            r_{\zeta^*}(x, y)=l_{\zeta^*}(x, y)=m>0.
        \end{equation}
        Define 
        \begin{align}
            &y^-:=\sup\{y': l_{\zeta^*}(x,y')=0\}, \qquad y^+:=\inf\{y': r_{\zeta^*}(x,y')=0\},\\ 
            &x^-:=\sup\{x': r_{\zeta^*}(x',y)=0\}, \qquad x^+:=\inf\{x': l_{\zeta^*}(x',y)=0\}.
        \end{align}
        By \cite[Lemma 4.7]{Delon_2010}, the map $l_{\zeta^*}$ (resp. $r_{\zeta^*}$) is continuous, monotonically decreasing in its first (second) argument, and continuous, monotonically increasing in its second (first) argument. Therefore, since $m>0$, $y^-<y<y^+$ and $x^-<x<x^+$. It means that equal and non-zero mass is moved from $(x^-, x)$ to $(y,y^+)$ and from $(x, x^+)$ to $(y^-, y)$. 
        By construction of $x^\pm,y^\pm$, we can also argue that there are support points $(x^-,y')$ with $y'\in(y, y^+)$, $(x^+,y')$ with $y'\in (y^-,y]$, $(x',y^-)$ with $x'\in(x,x ^+)$, and $(x',y^+)$ with $x'\in(x^-, x]$, like in Figure \ref{fig:placeholder}. 
        This is because $x^-$ is a supremum threshold where $r_{\zeta^*}(\cdot,y)$ becomes positive, there exist sequences approaching $x^-$ with positive mass in the corresponding rectangle, hence the support intersects the boundary line. Same for others.
        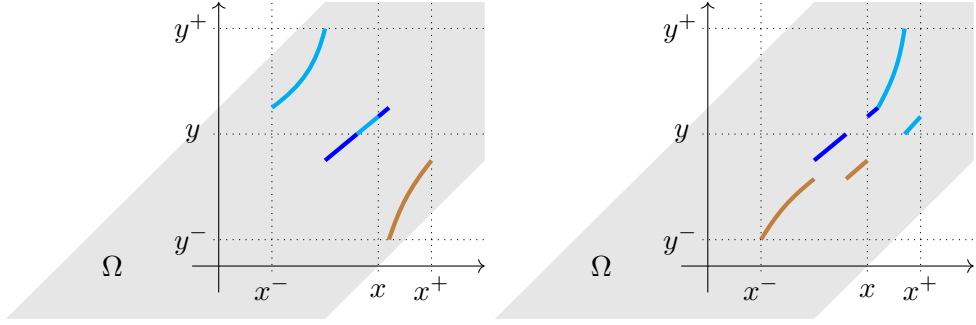
\begin{figure}[!ht]
            \centering
            \begin{tikzpicture}[scale=0.7]
\fill[fill=black!10] (-4,-1)--(0,3)--(2,5)--(5,5)--(5,2)--(3,0)--(2,-1);
\draw[->] (-0.5,0)--(5,0);
\draw[->] (0,-0.5)--(0,5);
\draw (1,0) node[below]{$x^-$}
(4,0) node[below]{$x^+$}
(3,-0.15) node[below]{$x$}
(0,0.5) node[left]{$y^-$}
(0,4.5) node[left]{$y^+$}
(-0.15,2.5) node[left]{$y$}
(-2, 0) node{$\Omega$};
\draw[dotted] (1,-0.1)--(1,5)
(4,-.1)--(4,5)
(3,-.1)--(3,5)
(-.1,4.5)--(5,4.5)
(-.1,2.5)--(5,2.5)
(-.1,.5)--(5,.5);
\draw[cyan, line width=1.6](1,3) to[bend right=20] (2,4.5)
(2.6,2.5) to (3, 2.83);
\draw[blue, line width=1.6] (2,2) to (2.6,2.5)
(3, 2.83)--(3.2,3);
\draw[brown, line width=1.6](3.2, 0.5) to[bend left=10] (4,2);
\end{tikzpicture}
\begin{tikzpicture}[scale=0.7]
\fill[fill=black!10] (-4,-1)--(0,3)--(2,5)--(5,5)--(5,2)--(3,0)--(2,-1);
\draw[->] (-0.5,0)--(5,0);
\draw[->] (0,-0.5)--(0,5);
\draw (1,0) node[below]{$x^-$}
(4,0) node[below]{$x^+$}
(3,-0.15) node[below]{$x$}
(0,0.5) node[left]{$y^-$}
(0,4.5) node[left]{$y^+$}
(-0.15,2.5) node[left]{$y$}
(-2, 0) node{$\Omega$};
\draw[dotted] (1,-0.1)--(1,5)
(4,-.1)--(4,5)
(3,-.1)--(3,5)
(-.1,4.5)--(5,4.5)
(-.1,2.5)--(5,2.5)
(-.1,0.5)--(5,0.5);
\draw[brown, line width=1.6](1,0.5) to[bend left=10] (2,1.65)
(2.6,1.65) to (3,2);
\draw[blue, line width=1.6]
(2,2) to (2.6,2.5)
(3, 2.83)--(3.2,3);
\draw[cyan, line width=1.6](3.2, 3) to[bend right=10] (3.7,4.5)
(3.7,2.5) to (4,2.83);
\end{tikzpicture}
            \caption{An example of the local modification construction, where the left and right picture depict the support of $\zeta^*$, and $\zeta_x$, respectively. The supports of $\tau_r$ on the left and of $\bar\tau_r$ on the right are highlighted in cyan. The supports of $\tau_l$ on the left and of $\bar\tau_l$ on the right are highlighted in brown. }
            \label{fig:placeholder}
        \end{figure}
        From this, we infer that $|x-y^-|<\Lambda$ and $|y-x^-|<\Lambda$, and it follows that the open rectangle $\operatorname{SW}:=(x^-,x)\times (y^-, y)$ is completely contained in $\Omega$. Similarly for $\operatorname{NE}:=(x,x^+)\times (y, y^+)$.
        The other two open rectangles $\operatorname{NW,SE}$ are not necessarily contained in $\Omega$, but they clearly have non-zero intersection with it.
        Define for $A\subseteq (x^-, x^+)$ and  $B\subseteq (y^-,y^+)$,
        \begin{align}
            &\tau_r(A\times B):=\zeta^*([A\cap (x^-,x)]\times [B\cap(y,y^+)]);\\
            &\tau_l(A\times B):=\zeta^*([A\cap (x,x^+)]\times [B\cap(y^-,y)]).
        \end{align}
        These are just the restrictions of $\zeta^*$ to the NW and SE rectangles respectively.
        Their marginals are denoted by $\tau_r^i$ and $\tau_l^i$ with $i=0,1$. 
        Now, the maps $l_y(\cdot):=l_{\zeta^*}(\cdot,y)$ and $r_y(\cdot):=r_{\zeta^*}(\cdot,y)$ are cumulative distribution functions of $\tau_l^0$ and $\tau_r^0$, respectively. Via a similar reasoning to \cite[Lemma 4.4]{Delon_2010} or our Lemma \ref{lemma: equivalence Gauss}, there are couplings $\chi_r\in \Gamma_b(\dd t\mres[0,m], \tau_r^1)$ and $\chi_l\in \Gamma_b(\dd t\mres[0,m], \tau_l^1)$ such that 
        \begin{align}
            &\tau_r(A\times B)=\chi_r(r_y(A\cap (x^-,x))\times [B\cap(y,y^+)]);\\
            &\tau_l(A\times B)=\chi_l(l_y(A\cap (x,x^+))\times [B\cap(y^-,y)]).
        \end{align}
        Then, we interchange the origin of the mass by defining
        \begin{align}
            &\bar\tau_r(A\times B):=\chi_r(l_y(A\cap (x,x^+))\times [B\cap(y,y^+)]);\\
            &\bar\tau_l(A\times B):=\chi_l(r_y(A\cap (x^-,x))\times [B\cap(y^-,y)]),
        \end{align}
        and finally, 
        \begin{equation}
            \zeta_x(A\times B):=\zeta^*(A\times B)-\tau_r(A\times B)-\tau_l(A\times B)+\bar{\tau}_r(A\times B)+\bar\tau_l(A\times B).
        \end{equation}
        The coupling $\zeta_x$ has the same uniform marginals, and it satisfies the two other properties of local modifications. The remaining part of the proof of \cite[Lemma 4.9]{Delon_2010}, where the authors prove the claim about the cost, is valid in our setting since $\tilde{c}$ is finite on $\supp(\zeta^*)$, and so all four terms appearing in the Monge comparisons are finite.
    \end{proof}
\end{lemma}
The previous type of construction can be used recursively to create cost equivalent local modifications that are supported on the graph of the monotone function $w_{\zeta^*}$ on closed intervals.
\begin{lemma}\label{lemma: int-mod}
    For every $x'<x''$, there exists a cost equivalent local modification $\zeta_{[x',x'']}$ of $\zeta^*$ such that $w_{\zeta_{[x',x'']}}(x)=w_{\zeta^*}(x)$ for all $x\in [x',x'']$, $m_{\zeta_{[x',x'']}}(x')=m_{\zeta_{[x',x'']}}(x'')=0$, and $\zeta_{[x',x'']}\mres ([x',x''] \times \R)$ is supported on the graph of $w_{\zeta^*}$.
    \begin{proof}
        Let $\{x_i\}_i\subset [x', x'']$ be a dense and countable subset including its endpoints. The proof \cite[Lemma 4.10]{Delon_2010} applies the previous lemma recursively at every point $x_i$, and this construction still applies under our assumptions.
    \end{proof}
\end{lemma}

\begin{proposition}
Under the above assumptions,
$\zeta^*$ is $\tilde c$-equivalent to a shift plan
\[
\zeta^\theta_i := (\Id\times(\Id+\theta))_\# dt\mres [-i,i] + \zeta^*\mres (-\infty, -i)\times (-\infty, -i+\theta)\cup (i, +\infty)\times (i+\theta, \infty),
\]
for every $i\in \R_+$.
\begin{proof}
A consequence of Lemma \ref{lemma: int-mod} is \cite[Corollary 4.11]{Delon_2010}, which ensures that there exists $\theta\in \R$ such that $w_{\zeta^*}(x)=x+\theta$. 
    Then, for $i\in \R_+$, let $\zeta_i:=\zeta_{[-i,i]}$ be the local modification in Lemma \ref{lemma: int-mod}. On each interval $[-i,i]$, $\zeta_i$ is supported on the graph of $w_{\zeta^*}= \Id+\theta$ and it is cost-equivalent.  
\end{proof}
\end{proposition}
{\section{Wasserstein-Fisher-Rao on discrete measures}\label{sec: app2}}
{We recall one finite-dimensional form of the Wasserstein-Fisher-Rao (WFR) distance and derive the explicit formula used in Lemma \ref{lemma: counter}.}

{Let $\mu_0=\sum_{i=1}^I r_0^i\delta_{x_0^i}$ and $\mu_1=\sum_{j=1}^J r_1^j\delta_{x_1^j}$ be two finitely-supported positive measures on $S^1$. The set of homogeneous couplings $\Gamma(\mu_0,\mu_1)$ contains measures $\gamma\in \calM_+(\bC\times \bC)$ of the type
\begin{equation}\label{eq:at-c}
    \gamma=\sum_{i=1}^I \delta_{(\gamma_0^{i0}x_0^i, 0)}+ \sum_{j=1}^J \delta_{(0, \gamma_1^{0j}x_1^j)}+ \sum_{i,j=1}^{I,J} \delta_{(\gamma_0^{ij}x_0^i, \gamma_1^{ij}x_1^j)},
\end{equation}
satisfying the following conditions
\begin{enumerate}[label=\textbf{C\arabic*}]
    \item \label{it: coupC1} for all $i,j$, it holds $\gamma_0^{ij}\geq 0$ and $\gamma_1^{ij}\geq 0$;
    \item \label{it: coupC2} for all $i=1,...,I$, it holds $r_0^i=\sum_{j=0}^J (\gamma_0^{ij})^2$;
    \item \label{it: coupC3} for all $j=1,...,J$, it holds $r_1^j=\sum_{i=0}^I (\gamma_1^{ij})^2$.
\end{enumerate}
The first condition reflects the proper decomposition in polar coordinates, and the other two are the marginal constraints. We will denote the space of such couplings as $\Gamma_a(\mu_0,\mu_1)$.
\begin{lemma}
    Let $\mu_0=\sum_{i=1}^I r_0^i\delta_{x_0^i}$ and $\mu_1=\sum_{j=1}^J r_1^j\delta_{x_1^j}$ be as above. Then, there exists a optimal coupling in $\Gamma_a(\mu_0,\mu_1)$ such that 
    \begin{enumerate}
        \item\label{it: coup1} If $\operatorname{dist}(x_0^i,x_1^j)\geq \pi/2$, then $\gamma_0^{ij}=\gamma_1^{ij}=0$;
        \item\label{it: coup2} If there exists $j$ such that $\operatorname{dist}(x_0^i,x_1^j)< \pi/2$, then $\gamma_0^{i0}=0$;
        \item\label{it: coup3} If there exists $i$ such that $\operatorname{dist}(x_0^i,x_1^j)< \pi/2$, then $\gamma_1^{0j}=0$.
    \end{enumerate}
    Consequently, 
    \begin{equation}\label{eq: fullwfr}
        \calU^2(\mu_0,\mu_1)=|\mu_0|(S^1) + |\mu_1|(S^1) -2\sup_{\gamma\in \Gamma_a} \sum_{i=1}^I \sum_{j=1}^J \gamma_0^{ij}\gamma_1^{ij}\cos(\operatorname{dist}(x_0^i,x_1^j)).
    \end{equation}
\end{lemma}
\begin{proof}
Let $c_{ij}:=\cos(\operatorname{dist}(x_0^i,x_1^j))$.
For any $\gamma\in\Gamma(\mu_0,\mu_1)$, the marginal constraints imply that its restriction to pairs of
nonzero coordinates is concentrated on the sets
\begin{equation}
    E_{ij}:=\{(s x_0^i,t x_1^j):s,t>0\}.
\end{equation}
Denote by $A_{ij}$ and $B_{ij}$ the integrals of $|z|^2$ and $|w|^2$ over $E_{ij}$, respectively.
By Cauchy--Schwarz,
\begin{equation}
\int_{E_{ij}}|z||w|\,d\gamma
\leq \sqrt{A_{ij}B_{ij}}.
\end{equation}
The homogeneous marginals of $\gamma\mres {E_{ij}}$ depend only on $A_{ij}$ and $B_{ij}$, while its energy can be obtained by relating the Euclidean distance in $\mathbb{C}$ with the one in $S^1$ via the cosine formula as
\begin{equation}
J(\gamma\mres {E_{ij}}) =A_{ij}+B_{ij} -2c_{ij}\int_{E_{ij}}|z||w|\,d\gamma.
\end{equation}
Thus, whenever $c_{ij}>0$, replacing $\gamma\mres {E_{ij}}$ by $\delta_{\left(\sqrt{A_{ij}}x_0^i,\sqrt{B_{ij}}x_1^j\right)}$ preserves the homogeneous marginals and does not increase the energy. When $c_{ij}\leq0$, the same holds by replacing it with $\delta_{\left(\sqrt{A_{ij}}x_0^i,0\right)} +\delta_{\left(0,\sqrt{B_{ij}}x_1^j\right)}$, that is, aggregating the remaining mass on the coordinate axes yields a coupling still denoted $\gamma \in \Gamma_a(\mu_0,\mu_1)$ with no larger energy.
Since
\begin{equation}
J(\gamma)=|\mu_0|(S^1)+|\mu_1|(S^1)
-2\int_{\bC\times\bC}\operatorname{Re}(z\bar w)
\,d\gamma,
\end{equation}
we obtain \eqref{eq: fullwfr}. The constraints \eqref{it: coupC1}-\eqref{it: coupC3} define a compact set of coefficients, so the supremum is attained.
By the preceding construction, an optimal coupling can be chosen with $\gamma_0^{ij}=\gamma_1^{ij}=0$ whenever $c_{ij}\leq0$, proving point \ref{it: coup1}. For point \ref{it: coup2}, suppose $c_{ij}>0$ for some $j$.
Replacing
\begin{equation}
\gamma_0^{ij}\longmapsto \sqrt{(\gamma_0^{ij})^2+(\gamma_0^{i0})^2},\qquad \gamma_0^{i0}\longmapsto 0
\end{equation}
preserves \eqref{it: coupC2} and cannot decrease the supremum in \eqref{eq: fullwfr}. The analogous replacement
\begin{equation}
\gamma_1^{ij}\longmapsto \sqrt{(\gamma_1^{ij})^2+(\gamma_1^{0j})^2}, \qquad \gamma_1^{0j}\longmapsto 0
\end{equation}
preserves \eqref{it: coupC3} and proves point \ref{it: coup3}. These modifications preserve point \ref{it: coup1} and can be performed for every applicable $i$ and $j$, yielding an optimal coupling satisfying all three properties.
\end{proof}}

{\begin{corollary}\label{cor:discrete-dirac-wfr}
Let $\mu_0=\sum_{i=1}^I r_0^i\delta_{x_0^i}$ and $\mu_1=m\delta_{x}$, with $m>0$.
Then 
\begin{equation}\label{eq:single-dirac-wfr}
\calU^2(\mu_0,\mu_1) =\sum_{i=1}^I r_0^i+m -2\sqrt{m\sum_{i=1}^I r_0^i \cos_+^2(\operatorname{dist}(x_0^i,x))},
\end{equation}
where $\cos_+(t):=\max\{\cos(t),0\}$.
\end{corollary}
\begin{proof}
Let $d_i:=\operatorname{dist}(x_0^i, x)$ and $p_i:=[\cos_+(d_i)]^2$. By \eqref{eq: fullwfr}, Cauchy--Schwarz and \eqref{it: coupC2} with $J=1$ and \eqref{it: coupC3},
\begin{equation}
\begin{aligned}
\sum_{i=1}^I \gamma_0^{i1}\gamma_1^{i1} \cos(d_i) &\leq \left(\sum_{i=1}^I r_0^ip_i\right)^{1/2} \left(\sum_{i=1}^I(\gamma_1^{i1})^2\right)^{1/2} \leq \sqrt{m\sum_{i=1}^I r_0^ip_i},
\end{aligned}
\end{equation}
Equality is attained by taking $\gamma_0^{i1}=\sqrt{r_0^i}$ for $p_i>0$, choosing $\gamma_1^{i1}$ proportional to $\sqrt{r_0^i}\cos_+(d_i)$, and any $\gamma_1^{i1}$ satisfying $\sum_i(\gamma_1^{i1})^2=m$. Note that this is feasible when there exists at least one point $x_0^i$ with $\operatorname{dist}(x_0^i, x)<\pi/2$, making $\gamma_1^{01}=0$. 
The remaining source mass is assigned to $\gamma_0^{i0}$. If all cosines are nonpositive, falling into item \ref{it: coup1}, the supremum is zero, attained by coupling through the coordinate axes.
\end{proof}}

{In particular, taking $\nu=\delta_{x}$ and $A=|\mu_0|(S^1)$, Definition \ref{def: homogenized} gives
\begin{equation}\label{eq:single-dirac-hom}
\bS_\nu(\mu_0) =\frac12\calU^2(\mu_0,A\nu) =A-\sqrt{A\sum_{i=1}^I r_0^ip_i}.
\end{equation}}

\small
\bibliographystyle{plain}
\bibliography{hkpolygons}

@article {Bru16,
    AUTHOR = {Bruveris, Martins},
     TITLE = {Optimal reparametrizations in the square root velocity
              framework},
   JOURNAL = {SIAM J. Math. Anal.},
  FJOURNAL = {SIAM Journal on Mathematical Analysis},
    VOLUME = {48},
      YEAR = {2016},
    NUMBER = {6},
     PAGES = {4335--4354},
      ISSN = {0036-1410,1095-7154},
   MRCLASS = {58B20 (58D15)},
  MRNUMBER = {3584579},
MRREVIEWER = {Alexander\ Schmeding},
       DOI = {10.1137/15M1014693},
       URL = {https://doi.org/10.1137/15M1014693},
}

@article{ SriKlaJosJer11,
  doi = {10.1109/tpami.2010.184},
  url = {https://doi.org/10.1109/tpami.2010.184},
  year = {2011},
  publisher = {Institute of Electrical and Electronics Engineers ({IEEE})},
  volume = {33},
  number = {7},
  pages = {1415--1428},
  author = {Srivastava, Anuj and Klassen, Eric and Joshi Shantanu H. and Jermyn, Ian H.},
  title = {Shape Analysis of Elastic Curves in {E}uclidean Spaces},
  fjournal = {{IEEE} Transactions on Pattern Analysis and Machine Intelligence},
  journal = {IEEE Trans. Pattern Anal. Mach. Intell.}
}

@article { BauHarKlas22,
    AUTHOR = {Bauer, Martin and Hartman, Emmanuel and Klassen, Eric},
     TITLE = {The square root normal field distance and unbalanced optimal
              transport},
   JOURNAL = {Appl. Math. Optim.},
  FJOURNAL = {Applied Mathematics and Optimization},
    VOLUME = {85},
      YEAR = {2022},
    NUMBER = {3},
     PAGES = {Paper No. 22, 40},
      ISSN = {0095-4616,1432-0606},
   MRCLASS = {49Q22 (49Q10)},
  MRNUMBER = {4419338},
MRREVIEWER = {Marc\ Sedjro},
       DOI = {10.1007/s00245-022-09867-y},
       URL = {https://doi.org/10.1007/s00245-022-09867-y},
}

@article {charon2020lengthmeasuresplanarclosed,
    AUTHOR = {Charon, Nicolas and Pierron, Thomas},
     TITLE = {On length measures of planar closed curves and the comparison
              of convex shapes},
   JOURNAL = {Ann. Global Anal. Geom.},
  FJOURNAL = {Annals of Global Analysis and Geometry},
    VOLUME = {60},
      YEAR = {2021},
    NUMBER = {4},
     PAGES = {863--901},
      ISSN = {0232-704X,1572-9060},
   MRCLASS = {53A04 (28A75 49Q20 49Q22)},
  MRNUMBER = {4328041},
MRREVIEWER = {Alina\ Stancu},
       DOI = {10.1007/s10455-021-09795-0},
       URL = {https://doi.org/10.1007/s10455-021-09795-0},
}

@article {michor2006riemannian,
    AUTHOR = {Michor, Peter and Mumford, David},
     TITLE = {Riemannian geometries on spaces of plane curves},
   JOURNAL = {J. Eur. Math. Soc. (JEMS)},
  FJOURNAL = {Journal of the European Mathematical Society (JEMS)},
    VOLUME = {8},
      YEAR = {2006},
    NUMBER = {1},
     PAGES = {1--48},
      ISSN = {1435-9855,1435-9863},
   MRCLASS = {58B20 (58D15)},
  MRNUMBER = {2201275},
MRREVIEWER = {Nikolai\ K.\ Smolentsev},
       DOI = {10.4171/JEMS/37},
       URL = {https://doi.org/10.4171/JEMS/37},
}

@article {gallouet2024regularitytheorygeometryunbalanced,
    AUTHOR = {Gallou\"et, Thomas and Ghezzi, Roberta and Vialard, Fran\c
              cois-Xavier},
     TITLE = {Regularity theory and geometry of unbalanced optimal
              transport},
   JOURNAL = {J. Funct. Anal.},
  FJOURNAL = {Journal of Functional Analysis},
    VOLUME = {289},
      YEAR = {2025},
    NUMBER = {7},
     PAGES = {Paper No. 111042, 54},
      ISSN = {0022-1236,1096-0783},
   MRCLASS = {49Q22},
  MRNUMBER = {4903404},
       DOI = {10.1016/j.jfa.2025.111042},
       URL = {https://doi.org/10.1016/j.jfa.2025.111042},
}

@article{bauer2024elastic,
  title={Elastic metrics on spaces of euclidean curves: Theory and algorithms},
  author={Bauer, Martin and Charon, Nicolas and Klassen, Eric and Kurtek, Sebastian and Needham, Tom and Pierron, Thomas},
  fjournal={Journal of Nonlinear Science},
  journal={J. Nonlinear Sci.},
  volume={34},
  number={3},
  pages={56},
  year={2024},
  publisher={Springer}
}

@article{MichorSRNF,
  title={Closed surfaces with different shapes that are indistinguishable by the {SRNF}},
  author={Michor, Peter and Klassen, Eric},
  fjournal={Archivum Mathematicum},
  journal={Arch. Math. (Brno)},
  volume={56},
  number={2},
  pages={107-114},
  year={2020},
  publisher={Masaryk University}
}

@incollection{hundrieser2021statisticscircularoptimaltransport,
    title={The Statistics of Circular Optimal Transport}, 
    author={Shayan Hundrieser and Marcel Klatt and Axel Munk},
    year={2022},
    booktitle={Directional Statistics for Innovative Applications},
    editor = {Ashis SenGupta and Barry C. Arnold},
    publisher={Springer Nature Singapore},
    pages={57-82},
}

@InProceedings{SriSRNF,
author="Jermyn, Ian H.
and Kurtek, Sebastian
and Klassen, Eric
and Srivastava, Anuj",
title="Elastic Shape Matching of Parameterized Surfaces Using Square Root Normal Fields",
booktitle="Computer Vision -- ECCV 2012",
year="2012",
publisher="Springer, Berlin, Heidelberg",
noaddress="Berlin, Heidelberg",
pages="804--817",
isbn="978-3-642-33715-4"
}

@article {SavSod23,
    AUTHOR = {Savar\'e, Giuseppe and Sodini, Giacomo Enrico},
     TITLE = {A relaxation viewpoint to unbalanced optimal transport:
              duality, optimality and {M}onge formulation},
   JOURNAL = {J. Math. Pures Appl. (9)},
  FJOURNAL = {Journal de Math\'ematiques Pures et Appliqu\'ees. Neuvi\`eme
              S\'erie},
    VOLUME = {188},
      YEAR = {2024},
     PAGES = {114--178},
      ISSN = {0021-7824,1776-3371},
   MRCLASS = {49Q22 (28A33 49K27)},
  MRNUMBER = {4756504},
MRREVIEWER = {Danka\ Lu\v ci\'c},
       DOI = {10.1016/j.matpur.2024.05.009},
       URL = {https://doi.org/10.1016/j.matpur.2024.05.009},
}

@misc{chizat2019unbalancedoptimaltransportdynamic,
      title={Unbalanced Optimal Transport: Dynamic and Kantorovich Formulation}, 
      author={Lenaic Chizat and Gabriel Peyré and Bernhard Schmitzer and François-Xavier Vialard},
      year={2019},
      eprint={1508.05216},
      archivePrefix={arXiv},
      primaryClass={math.OC},
      url={https://arxiv.org/abs/1508.05216}, 
}

@article{Eckhardt_2019,
   title={Elastic energy regularization for inverse obstacle scattering problems},
   volume={35},
   ISSN={1361-6420},
   url={http://dx.doi.org/10.1088/1361-6420/ab3034},
   DOI={10.1088/1361-6420/ab3034},
   number={10},
   journal={Inverse Problems},
   fjournal={Inverse Problems},
   publisher={IOP Publishing},
   author={Eckhardt, Julian and Hiptmair, Ralf and Hohage, Thorsten and Schumacher, Hendrik and Wardetzky, Max},
   year={2019},
   pages={104009} }

@article {HaBaKla24,
    AUTHOR = {Hartman, Emmanuel and Bauer, Martin and Klassen, Eric},
     TITLE = {Square root normal fields for {L}ipschitz surfaces and the
              {W}asserstein {F}isher {R}ao metric},
   JOURNAL = {SIAM J. Math. Anal.},
  FJOURNAL = {SIAM Journal on Mathematical Analysis},
    VOLUME = {56},
      YEAR = {2024},
    NUMBER = {2},
     PAGES = {2171--2190},
      ISSN = {0036-1410,1095-7154},
   MRCLASS = {49Q10 (49Q22)},
  MRNUMBER = {4715282},
MRREVIEWER = {Qinfeng\ Li},
       DOI = {10.1137/22M1544452},
       URL = {https://doi.org/10.1137/22M1544452},
}

@article{IwaOnni09,
author = {Iwaniec, Tadeusz and Onninen, Jani},
year = {2009},
pages = {927-986},
title = {Hyperelastic Deformations of Smallest Total Energy},
volume = {194},
number={3},
fjournal = {Archive for Rational Mechanics and Analysis},
journal = {Arch. Ration. Mech. Anal.},
doi = {10.1007/s00205-008-0192-7}
}

@article {Younes_2008,
    AUTHOR = {Younes, Laurent and Michor, Peter and Shah, Jayant and
              Mumford, David},
     TITLE = {A metric on shape space with explicit geodesics},
   JOURNAL = {Atti Accad. Naz. Lincei Rend. Lincei Mat. Appl.},
  FJOURNAL = {Atti della Accademia Nazionale dei Lincei. Rendiconti Lincei.
              Matematica e Applicazioni},
    VOLUME = {19},
      YEAR = {2008},
    NUMBER = {1},
     PAGES = {25--57},
      ISSN = {1120-6330,1720-0768},
   MRCLASS = {58D10 (58B20 94A08)},
  MRNUMBER = {2383560},
MRREVIEWER = {Fran\c{c}ois\ Gay-Balmaz},
       DOI = {10.4171/RLM/506},
       URL = {https://doi.org/10.4171/RLM/506},
}

@article {Needham_2020,
    AUTHOR = {Needham, Tom and Kurtek, Sebastian},
     TITLE = {Simplifying transforms for general elastic metrics on the
              space of plane curves},
   JOURNAL = {SIAM J. Imaging Sci.},
  FJOURNAL = {SIAM Journal on Imaging Sciences},
    VOLUME = {13},
      YEAR = {2020},
    NUMBER = {1},
     PAGES = {445--473},
      ISSN = {1936-4954},
   MRCLASS = {58B20 (58E50 68U05)},
  MRNUMBER = {4075332},
MRREVIEWER = {Jiansong\ Deng},
       DOI = {10.1137/19M1265132},
       URL = {https://doi.org/10.1137/19M1265132},
}

@article{Liero_2017,
   title={Optimal Entropy-Transport problems and a new {H}ellinger–{K}antorovich distance between positive measures},
   volume={211},
   ISSN={1432-1297},
   url={http://dx.doi.org/10.1007/s00222-017-0759-8},
   DOI={10.1007/s00222-017-0759-8},
   number={3},
   fjournal={Inventiones Mathematicae},
   journal={Invent. Math.},
   publisher={Springer Science and Business Media LLC},
   author={Liero, Matthias and Mielke, Alexander and Savaré, Giuseppe},
   year={2017},
   pages={969–1117} }

@book {Santa,
    AUTHOR = {Santambrogio, Filippo},
     TITLE = {Optimal transport for applied mathematicians},
    SERIES = {Progress in Nonlinear Differential Equations and their
              Applications},
    VOLUME = {87},
 PUBLISHER = {Birkh\"auser/Springer, Cham},
      YEAR = {2015},
     PAGES = {xxvii+353},
      ISBN = {978-3-319-20827-5; 978-3-319-20828-2},
   MRCLASS = {49-02 (35J96 49J45 49M29 58E50 90C05 90C48 91B02)},
  MRNUMBER = {3409718},
MRREVIEWER = {Luigi\ De Pascale},
       DOI = {10.1007/978-3-319-20828-2},
       URL = {https://doi.org/10.1007/978-3-319-20828-2},
}

@inproceedings{bonet2023sphericalslicedwasserstein,
title={Spherical Sliced-{W}asserstein},
author={Cl{\'e}ment Bonet and Paul Berg and Nicolas Courty and Fran{\c{c}}ois Septier and Lucas Drumetz and Minh Tan Pham},
booktitle={The Eleventh International Conference on Learning Representations },
year={2023},
url={https://openreview.net/forum?id=jXQ0ipgMdU}
}

@article{Delon_2010,
   title={Fast Transport Optimization for Monge Costs on the Circle},
   volume={70},
   ISSN={1095-712X},
   url={http://dx.doi.org/10.1137/090772708},
   DOI={10.1137/090772708},
   number={7},
   fjournal={SIAM Journal on Applied Mathematics},
   journal={SIAM J. Appl. Math.},
   publisher={Society for Industrial & Applied Mathematics (SIAM)},
   author={Delon, Julie and Salomon, Julien and Sobolevski, Andrei},
   year={2010},
   pages={2239–2258} }

@article{Cristinelli_2025,
   title={Conditional gradients for total variation regularization with {PDE} constraints: a graph cuts approach},
   ISSN={1573-2894},
   url={http://dx.doi.org/10.1007/s10589-025-00699-4},
   DOI={10.1007/s10589-025-00699-4},
   fjournal={Computational Optimization and Applications},
   journal={Comput. Optim. Appl.},
   publisher={Springer Science and Business Media LLC},
   author={Cristinelli, Giacomo and Iglesias, José A. and Walter, Daniel},
   year={2026},
   volume={93},
   pages={209-265}}

@unpublished{cristinelli2025linearconvergenceonecutconditional,
      title={Linear convergence of a one-cut conditional gradient method for total variation regularization}, 
      author={Giacomo Cristinelli and José A. Iglesias and Daniel Walter},
      year={2025},
      note={Preprint arXiv 2504.16899 [math.OC]},
      eprint={2504.16899},
      archivePrefix={arXiv},
      primaryClass={math.OC},
      url={https://arxiv.org/abs/2504.16899}, 
}

@article {YueKuWie,
    AUTHOR = {Yue, Man-Chung and Kuhn, Daniel and Wiesemann, Wolfram},
     TITLE = {On linear optimization over {W}asserstein balls},
   JOURNAL = {Math. Program.},
  FJOURNAL = {Mathematical Programming},
    VOLUME = {195},
      YEAR = {2022},
    NUMBER = {1-2},
     PAGES = {1107--1122},
      ISSN = {0025-5610,1436-4646},
   MRCLASS = {90C25 (49Q22 90C17 90C48)},
  MRNUMBER = {4499079},
MRREVIEWER = {Hao\ Wu},
       DOI = {10.1007/s10107-021-01673-8},
       URL = {https://doi.org/10.1007/s10107-021-01673-8},
}

@article { CarIglWal25,
    AUTHOR = {Carioni, Marcello and Iglesias, Jos\'{e} A. and Walter, Daniel},
     TITLE = {Extremal {P}oints and {S}parse {O}ptimization for
              {G}eneralized {K}antorovich--{R}ubinstein {N}orms},
   JOURNAL = {Found. Comput. Math.},
  FJOURNAL = {Foundations of Computational Mathematics. The Journal of the
              Society for the Foundations of Computational Mathematics},
    VOLUME = {25},
      YEAR = {2025},
    NUMBER = {1},
     PAGES = {103--144},
      ISSN = {1615-3375},
   MRCLASS = {49Q22 (46A55 52A40 65J22)},
  MRNUMBER = {4865120},
       DOI = {10.1007/s10208-023-09634-7},
}

@article{BreCarFanRom21,
  title={On the extremal points of the ball of the {B}enamou--{B}renier energy},
  author={Bredies, Kristian and Carioni, Marcello and Fanzon, Silvio and Romero, Francisco},
  fjournal={Bulletin of the London Mathematical Society},
  journal={Bull. Lond. Math. Soc.},
  volume={53},
  number={5},
  pages={1436--1452},
  year={2021},
  publisher={Wiley Online Library}
}

@article{BreCarFanRom23,
	author = {Bredies, Kristian and Carioni, Marcello and Fanzon, Silvio and Romero, Francisco},
	date = {2023/06/01},
	doi = {10.1007/s10208-022-09561-z},
	id = {Bredies2023},
	isbn = {1615-3383},
	journal = {Found. Comput. Math.},
    fjournal = {Foundations of Computational Mathematics},
	number = {3},
	pages = {833--898},
	title = {A Generalized Conditional Gradient Method for Dynamic Inverse Problems with Optimal Transport Regularization},
	url = {https://doi.org/10.1007/s10208-022-09561-z},
	volume = {23},
	year = {2023}
}

@article{BreCarFan22,
 author = {Bredies, Kristian and Carioni, Marcello and Fanzon, Silvio},
 title = {A superposition principle for the inhomogeneous continuity equation with {Hellinger}-{Kantorovich}-regular coefficients},
 fjournal = {Communications in Partial Differential Equations},
 journal = {Commun. Partial Differ. Equations},
 issn = {0360-5302},
 volume = {47},
 number = {10},
 pages = {2023--2069},
 year = {2022},
 language = {English},
 doi = {10.1080/03605302.2022.2109172},
 zbMATH = {7600438},
 Zbl = {1518.35226}
}

@article{LavBlaAub24,
 author = {Laville, Bastien and Blanc-F{\'e}raud, Laure and Aubert, Gilles},
 title = {A {{\(\varGamma\)}}-convergence result and an off-the-grid charge algorithm for curve reconstruction in inverse problems},
 fjournal = {Journal of Mathematical Imaging and Vision},
 journal = {J. Math. Imaging Vis.},
 issn = {0924-9907},
 volume = {66},
 number = {4},
 pages = {572--583},
 year = {2024},
 language = {English},
 doi = {10.1007/s10851-024-01190-1},
 zbMATH = {8014369},
 Zbl = {1561.68188}
}

@article { CasDuvPet22,
    AUTHOR = {De Castro, Yohann and Duval, Vincent and Petit, Romain},
     TITLE = {Towards off-the-grid algorithms for total variation regularized inverse problems},
   JOURNAL = {J. Math. Imaging Vision},
  FJOURNAL = {Journal of Mathematical Imaging and Vision},
    VOLUME = {65},
      YEAR = {2023},
    NUMBER = {1},
     PAGES = {53--81},
      ISSN = {0924-9907},
   MRCLASS = {68U10},
  MRNUMBER = {2910876},
       DOI = {10.1007/s10851-022-01115-w},
       URL = {https://doi.org/10.1007/s10851-022-01115-w},
}

@article{BreCarFanWal24,
 author = {Bredies, Kristian and Carioni, Marcello and Fanzon, Silvio and Walter, Daniel},
 title = {Asymptotic linear convergence of fully-corrective generalized conditional gradient methods},
 fjournal = {Mathematical Programming. Series A. Series B},
 journal = {Math. Program.},
 issn = {0025-5610},
 volume = {205},
 number = {1-2 (A)},
 pages = {135--202},
 year = {2024},
 language = {English},
 doi = {10.1007/s10107-023-01975-z},
 zbMATH = {7829603},
 Zbl = {1536.65058}
}

@article{DenDuvPeySou20,
 author = {Denoyelle, Quentin and Duval, Vincent and Peyr{\'e}, Gabriel and Soubies, Emmanuel},
 title = {The sliding {Frank}-{Wolfe} algorithm and its application to super-resolution microscopy},
 fjournal = {Inverse Problems},
 journal = {Inverse Probl.},
 issn = {0266-5611},
 volume = {36},
 number = {1},
 pages = {42},
 note = {Paper No. 014001},
 year = {2020},
 language = {English},
 doi = {10.1088/1361-6420/ab2a29},
 zbMATH = {7153864},
 Zbl = {1434.65082}
}

@unpublished{ CarDel23,
  doi = {10.48550/ARXIV.2311.08072},
  url = {https://arxiv.org/abs/22311.08072},
  author = {Carioni, Marcello and Del Grande, Leonardo},
  title = {A general theory for exact sparse representation recovery in convex optimization},
  note  = {Preprint arXiv:2311.08072 [math.OC]},
  publisher = {arXiv},
  type = {Preprint},
  year = {2023}
}

@article{GaoKle23,
 author = {Gao, Rui and Kleywegt, Anton},
 title = {Distributionally robust stochastic optimization with {Wasserstein} distance},
 fjournal = {Mathematics of Operations Research},
 journal = {Math. Oper. Res.},
 issn = {0364-765X},
 volume = {48},
 number = {2},
 pages = {603--655},
 year = {2023},
 language = {English},
 doi = {10.1287/moor.2022.1275},
 zbMATH = {7808961},
 Zbl = {1540.90171}
}

@article{ MohKuh18,
 author = {Mohajerin Esfahani, Peyman and Kuhn, Daniel},
 title = {Data-driven distributionally robust optimization using the {Wasserstein} metric: performance guarantees and tractable reformulations},
 fjournal = {Mathematical Programming. Series A. Series B},
 journal = {Math. Program.},
 issn = {0025-5610},
 volume = {171},
 number = {1-2 (A)},
 pages = {115--166},
 year = {2018},
 language = {English},
 doi = {10.1007/s10107-017-1172-1},
 url = {infoscience.epfl.ch/record/207778},
 zbMATH = {6945237},
 Zbl = {1433.90095}
}

@incollection{Schneider1993ConvexSurfaces,
  author    = {Schneider, Rolf},
  title     = {Convex Surfaces, Curvature and Surface Area Measures},
  booktitle = {Handbook of Convex Geometry},
  pages     = {273--299},
  year      = {1993},
  publisher = {Elsevier}
}

@unpublished{sellaroli2017algorithmreconstructconvexpolyhedra,
      title={An algorithm to reconstruct convex polyhedra from their face normals and areas}, 
      author={Giuseppe Sellaroli},
      year={2017},
      note={Preprint arXiv 1712.00825 [cs.CG]},
      eprint={1712.00825},
      archivePrefix={arXiv},
      primaryClass={cs.CG},
      url={https://arxiv.org/abs/1712.00825}, 
}

@misc{fdasrsf,
  author       = {Tucker, J. Derek},
  title        = {{fdasrsf}: Functional Data Analysis Using the Square Root Slope Framework},
  howpublished = {\url{https://fdasrsf-python.readthedocs.io/}},
  note         = {Python package, accessed 9 November 2026}
}

@article{flamary2021pot,
  author  = {R{\'e}mi Flamary and Nicolas Courty and Alexandre Gramfort and Mokhtar Z. Alaya and Aur{\'e}lie Boisbunon and Stanislas Chambon and Laetitia Chapel and Adrien Corenflos and Kilian Fatras and Nemo Fournier and L{\'e}o Gautheron and Nathalie T.H. Gayraud and Hicham Janati and Alain Rakotomamonjy and Ievgen Redko and Antoine Rolet and Antony Schutz and Vivien Seguy and Danica J. Sutherland and Romain Tavenard and Alexander Tong and Titouan Vayer},
  title   = {POT: Python Optimal Transport},
  journal = {Journal of Machine Learning Research},
  year    = {2021},
  volume  = {22},
  number  = {78},
  pages   = {1-8},
  url     = {http://jmlr.org/papers/v22/20-451.html}
}

@inproceedings{chapel2021unbalanced,
  title={Unbalanced Optimal Transport through Non-negative Penalized Linear Regression},
  author={Chapel, Laetitia and Flamary, R{\'e}mi and Wu, Haoran and F{\'e}votte, C{\'e}dric and Gasso, Gilles},
  booktitle={Advances in Neural Information Processing Systems 34},
  year={2021}
}
\normalsize

\end{document}